\documentclass[a4paper,12pt]{amsart}
\usepackage{amssymb}

\theoremstyle{plain}
  \newtheorem{thm}{Theorem}[section]
  \newtheorem{lem}[thm]{Lemma}
  \newtheorem{slem}[thm]{Sublemma}
  \newtheorem{cor}[thm]{Corollary}
  \newtheorem{prop}[thm]{Proposition}
  \newtheorem{fact}[thm]{Fact}

\theoremstyle{definition}
  \newtheorem{defn}[thm]{Definition}
  
  \newtheorem{asmp}[thm]{Assumption}

\theoremstyle{remark}
  \newtheorem{rem}[thm]{Remark}

\numberwithin{equation}{section}

\DeclareMathOperator{\vol}{vol}
\DeclareMathOperator{\diam}{diam}
\DeclareMathOperator{\divv}{div}
\DeclareMathOperator{\supp}{supp}

\DeclareMathOperator{\Ric}{Ric}

\DeclareMathOperator{\bdiv}{\overline{\text{\rm div}}}
\DeclareMathOperator{\blap}{\bar\Delta}
\DeclareMathOperator{\dbl}{dbl}
\DeclareMathOperator{\BG}{BG}

\newcommand{\Hm}{\mathcal{H}}
\newcommand{\loc}{{\text{\rm loc}}}

\newcommand{\Cut}{\text{\rm Cut}}

\newcommand{\field}[1]{\mathbb{#1}}

\newcommand{\R}{\field{R}}

\begin{document}

\title{Laplacian comparison\\ for Alexandrov spaces}

\date{September 6, 2007}

\begin{abstract}
  We consider an infinitesimal version of the Bishop-Gromov
  relative volume comparison condition
  as generalized notion of Ricci curvature bounded below
  for Alexandrov spaces.
  We prove a Laplacian comparison theorem for Alexandrov spaces under
  the condition.  As an application we prove a topological
  splitting theorem.
\end{abstract}

\author{Kazuhiro Kuwae}
\address{Department of Mathematics, Faculty of Education,
  Kumamoto University, Kumamoto, 860-8555, JAPAN}
\email{kuwae@gpo.kumamoto-u.ac.jp}

\author{Takashi Shioya}
\address{Mathematical Institute, Tohoku University, Sendai 980-8578,
  JAPAN}
\email{shioya@math.tohoku.ac.jp}

\dedicatory{Dedicated to Professor Karsten Grove on the occasion of
his sixtieth birthday.}




\thanks{The authors are partially supported by a Grant-in-Aid
  for Scientific Research No.~16540201 and 14540056 from
  the Japan Society for the Promotion of Science}

\maketitle

\section{Introduction} \label{sec:intro}

In this paper, we study singular spaces of
Ricci curvature bounded below.
For Riemannian manifolds, having a lower bound of
Ricci curvature is equivalent to an infinitesimal version
of the Bishop-Gromov volume comparison condition.
Since it is impossible to define the Ricci curvature tensor
on Alexandrov spaces,
we consider such the volume comparison condition
as a candidate of the conditions of the Ricci curvature bounded below.

In Riemannian geometry, the Laplacian comparison theorem
is one of the most important tools
to study the structure of spaces with a lower bound of Ricci curvature.
A main purpose of this paper is to prove
a Laplacian comparison theorem for Alexandrov spaces
under the volume comparison condition.
As an application, we prove a topological splitting theorem
of Cheeger-Gromoll type.

Let us present the volume comparison condition.
For $\kappa \in \R$, we set
\[
s_\kappa(r) :=
\begin{cases}
  \sin(\sqrt{\kappa}r)/\sqrt{\kappa}  &\text{if $\kappa > 0$},\\
  r &\text{if $k = 0$},\\
  \sinh(\sqrt{|\kappa|}r)/\sqrt{|\kappa|} &\text{if $\kappa < 0$}.
\end{cases}
\]
The function $s_\kappa$ is the solution of the Jacobi equation
$s_\kappa''(r) + \kappa s_\kappa(r) = 0$ with initial condition
$s_\kappa(0) = 0$, $s_\kappa'(0) = 1$.

Let $M$ be an Alexandrov space of dimension $n \ge 2$.
For $p \in M$ and $0 < t \le 1$,
we define a subset $W_{p,t} \subset M$ and a map $\Phi_{p,t} : W_{p,t} \to M$
as follows.
$x \in W_{p,t}$ if and only if there exists $y \in M$ such that
$x \in py$ and $d(p,x) : d(p,y) = t:1$, where
$py$ is a minimal geodesic from $p$ to $y$ and $d$ the distance function.
For a given point $x \in W_{p,t}$ such a point $y$ is unique and we set
$\Phi_{p,t}(x) := y$.  The Alexandrov convexity (cf.~\S\ref{ssec:Alex})
implies the Lipschitz continuity of the map $\Phi_{p,t}$.
Let us consider the following.

{\bf Condition $\BG(\kappa)$ at a point $p \in M$:}  We have
\[
d(\Phi_{p,t\,*} \Hm^n)(x)
\ge \frac{t\,s_\kappa(t\,d(p,x))^{n-1}}{s_\kappa(d(p,x))^{n-1}}
d\Hm^n(x)
\]
for any $x \in M$ and $t \in (\,0,1\,]$ such that
$d(p,x) < \pi/\sqrt{\kappa}$ if $\kappa > 0$,
where $\Phi_{p,t\,*}\Hm^n$ means the push-forward
by $\Phi_{p,t}$ of the $n$-dimensional Hausdorff measure $\Hm^n$ on $M$.

If $M$ satisfies $\BG(\kappa)$ at any point $p \in M$,
we simply say that $M$ satisfies $\BG(\kappa)$.

The condition, $\BG(\kappa)$, is an infinitesimal version of
the Bishop-Gromov inequality.  For an $n$-dimensional complete
Riemannian manifold, $\BG(\kappa)$ holds if and only if
the Ricci curvature satisfies $\Ric \ge (n-1)\kappa$
(see Theorem 3.2 of \cite{Ota:mcp} for the `only if' part).
We see some studies on similar (or same) conditions to $\BG(\kappa)$ in
\cite{CC:strRicI,St:heat,KwSy:geneMCP,KwSy:sobmet,Rm:Poincare,Ota:mcp,
Wt:localcut} etc.
$\BG(\kappa)$ is sometimes called the Measure Contraction Property
and is weaker than the curvature-dimension condition introduced
by Sturm \cite{St:geomI,St:geomII} and Lott-Villani \cite{LV:Ricmm}.
Any Alexandrov space of curvature $\ge \kappa$ satisfies $\BG(\kappa)$.
However we do not necessarily assume $M$ to be of curvature $\ge \kappa$.
For example, a Gromov-Hausdorff limit of closed $n$-manifolds of
$\Ric \ge (n-1)\kappa$, sectional curvature $\ge \kappa_0$, diameter $\le D$,
and volume $\ge v > 0$ is an Alexandrov space with $\BG(\kappa)$ and of
curvature $\ge \kappa_0$.

To state the Laplacian comparison theorem, we need some notations and
definitions.
If $M$ has no boundary, we define $M^*$ as
the set of non-$\delta$-singular points of $M$ for a
number $\delta$ with $0 < \delta \ll 1/n$.
If $M$ has nonempty boundary, we refer to Fact \ref{fact:str} below
for $M^*$.  All the topological singularities of $M$ are entirely
contained in $M \setminus M^*$ and
$M^*$ has a natural structure of $C^\infty$ differentiable manifold.
We have a canonical Riemannian metric $g$ on $M^*$
which is a.e.~continuous and of locally bounded variation
(locally BV for short).  See \S\ref{ssec:Alex} for more details.
We set $\cot_\kappa(r) := s_\kappa'(r)/s_\kappa(r)$ and
$r_p(x) := d(p,x)$ for $p,x \in M$.
The \emph{distributional Laplacian $\blap r_p$ of $r_p$ on $M^*$}
is defined by the usual formula:
\[
\blap r_p := -D_i(\sqrt{|g|} \, g^{ij} \, \partial_j r),
\]
on a local chart of $M^*$, where $D_i$ is the distributional derivative
with respect to the $i^{\rm th}$ coordinate.
Then, $\blap r_p$ becomes a signed Radon measure on $M^*$
(see \S\ref{sec:Green}).
An main theorem of this paper is stated as follows.

\begin{thm}[Laplacian Comparison Theorem] \label{thm:LapComp}
  Let $M$ be an Alexandrov space of dimension $n \ge 2$.
  If $M$ satisfies $\BG(\kappa)$ at a point $p \in M$, then we have
  \begin{equation}
    \label{eq:LapComp}
    d\blap r_p \ge -(n-1)\cot_\kappa \circ r_p\,d\Hm^n
    \quad\text{on $M^* \setminus \{p\}$}.
  \end{equation}
\end{thm}

\begin{cor} \label{cor:LapComp}
  If $M$ is an Alexandrov space of dimension $n \ge 2$ and
  curvature $\ge \kappa$,
  then \eqref{eq:LapComp} holds for any $p \in M$.
\end{cor}

Even if $M$ is a Riemannian manifold, $\blap r_p$ is not absolutely
continuous with respect to $\Hm^n$ on the cut-locus of $p$
(see Remark \ref{rem:LapCut}).
Different from Riemannian, the cut-locus
of an Alexandrov space is not necessarily a closed subset.
In fact, we have an example of an Alexandrov space for which
the singular set and the cut-locus are both dense in the space
(cf.~Example (2) in \S 0 of \cite{OS:rstralex}).
The Riemannian metric $g$ on $M^*$ is not continuous
on any singular point
and has at most the regularity of locally BV.
Therefore, the Laplacian of a $C^\infty$ function does not become
a function, only does a Radon measure in general.
In particular, considering a Laplacian comparison in the barrier sense
is meaningless.
In this reason, for Theorem \ref{thm:LapComp} a standard proof for Riemannian
does not work and we need a more delicate discussion using BV theory.

In \cite{Pt:subharmAlex}, Petrunin claims that the Laplacian of
any $\lambda$-semiconvex function is $\ge -n\lambda$ from the study of
gradient curves.  This implies Corollary \ref{cor:LapComp}.
However we do not know the details.
After Petrunin, Renesse \cite{Rs:compalex} proved Corollary \ref{cor:LapComp}
in a different way under some additional condition.
Our proof is based on a different idea from them.

We do not know if the converse to Theorem \ref{thm:LapComp} is true or not,
i.e., if \eqref{eq:LapComp} implies $\BG(\kappa)$ at $p$.
For $C^\infty$ Riemannian manifolds, this is easy to prove.

As an application to Theorem \ref{thm:LapComp} we have

\begin{thm}[Topological Splitting Theorem] \label{thm:splitting}
  If an Alexandrov space $M$ satisfies $\BG(0)$ and contains
  a straight line, then $M$ is homeomorphic to $N \times \R$
  for some topological space $N$.
\end{thm}

We do not know if the isometric splitting in the theorem is true,
i.e., if $M$ is isometric to $N \times \R$ for some Alexandrov space
$N$.
If we replace `$\BG(0)$' with `curvature $\ge 0$', then the isometric
splitting is well-known (\cite{Mk:line})
as a generalization of the Toponogov splitting theorem.
For Riemannian manifolds, $\BG(0)$ is equivalent to $\Ric \ge 0$
and the isometric splitting was proved by
Cheeger-Gromoll \cite{CG:split}.
In our case, we do not have the Weitzenb\"ock formula, so that
we cannot obtain the isometric splitting at present.

If the metric of $M$ has enough $C^\infty$ part,
we prove the isometric splitting.

\begin{cor} \label{cor:splitting}
  Let $M$ be an Alexandrov space.  Assume that the singular set of $M$
  is closed and the non-singular set is an {\rm(}incomplete{\rm)}
  $C^\infty$ Riemannian manifold of $\Ric \ge 0$.
  If $M$ contains a straight line,
  then $M$ is isometric to $N \times \R$ for some Alexandrov space $N$.
\end{cor}

For Riemannian orbifolds,
Borzellino-Zhu \cite{BZ:splitting} proved
an isometric splitting theorem.  Corollary \ref{cor:splitting}
is more general than their result.

In our previous paper \cite{KMS:lap},
we proved for an Alexandrov space $M$
the existence of the heat kernel of $M$ and the discreteness of
the spectrum of the generator (Laplacian) of the Dirichlet energy form
on a relatively compact domain in $M$.
As another application to Theorem \ref{thm:LapComp},
we have the following heat kernel
and first eigenvalue comparison results,
which generalize the results of Cheeger-Yau \cite{CY:heatkernel} and
Cheng \cite{Ch:eigencomp}.

$B(p,r)$ denotes the metric ball centered at $p$ and of radius $r$
and $M^n(\kappa)$ an $n$-dimensional
complete simply connected space form of curvature $\kappa$.

\begin{cor} \label{cor:heatcomp}
  Let $M$ be an $n$-dimensional Alexandrov space which satisfies
  $\BG(\kappa)$ at a point $p \in M$, and $\Omega \subset M$
  an open subset containing $B(p,r)$ for
  a number $r > 0$.
  Denote by $h_t : \Omega \times \Omega \to \R$, $t > 0$,
  the heat kernel on $\Omega$ with Dirichlet boundary condition,
  and by $\bar h_t : B(\bar p,r) \times B(\bar p,r) \to \R$
  that on $B(\bar p,r)$
  for a point $\bar p \in M^n(\kappa)$.
  Then, for any $t > 0$ and $q \in B(p,r)$ we have
  \[
  h_t(p,q) \ge \bar h_t(\bar p,\bar q),
  \]
  where $\bar q \in M^n(\kappa)$ is a point such that
  $d(\bar p,\bar q) = d(p,q)$.
\end{cor}

\begin{cor} \label{cor:eigen}
  Let $M$ be an $n$-dimensional Alexandrov space which satisfies
  $\BG(\kappa)$ at a point $p \in M$, and $r > 0$ a number.
  Denote by $\lambda_1(B(p,r))$ the first eigenvalue of
  the generator {\rm(}Laplacian{\rm)} of the Dirichlet energy form on
  $B(p,r)$ with Dirichlet boundary condition,
  and by $\lambda_1(B(\bar p,r))$ that on $B(\bar p,r)$ for
  a point $\bar p \in M^n(\kappa)$.
  Then we have
  \[
  \lambda_1(B(p,r)) \le \lambda_1(B(\bar p,r)).
  \]
\end{cor}

Once we have the Laplacian Comparison Theorem
(see Corollary \ref{cor:LapCompform}),
the proofs of Corollaries \ref{cor:heatcomp} and \ref{cor:eigen}
are the same as of Theorem II and Corollary 1 of Renesse's paper
\cite{Rs:compalex}.  
We can carefully verify that the local $(L^1,1)$-volume regularity
is not needed in the proof of Theorem II of \cite{Rs:compalex}.


We also obtain a Brownian motion comparison theorem
in the same way as in \cite{Rs:compalex}.  The detail is omitted here.

\begin{rem}
  All the results above are true even
  in the case where $M$ has non-empty boundary.
  In Corollaries \ref{cor:heatcomp} and \ref{cor:eigen},
  we implicitly assume the Neumann boundary condition on the boundary of $M$
  for the heat kernel and the first eigenvalue.
  In particular, the results hold for any convex subset of an
  Alexandrov space.
\end{rem}

Let us briefly mention the idea of the proof of Theorem \ref{thm:LapComp}.
One of the important steps is to prove
the Green formula on a region $E \subset M^*$ with piecewise smooth
boundary:
\[
\blap r_p(E) = \int_{\partial E} \langle \nu_E,\nabla r_p \rangle
\; d\Hm^{n-1},
\]
where $\nu_E$ is the inward normal vector
field along $\partial E$ of $E$ (Theorem \ref{thm:Green}).
For the proof of the Green formula, it is essential to prove that
$\bdiv_{g^{(h)}} Y \to \bdiv_g Y$ weakly $*$ as $h \to 0$
(Lemma \ref{lem:divconv}),
where $Y$ is any $C^\infty$ vector field on $M^*$,
$g^{(h)}$ the $C^\infty$ mollifier of the Riemannian metric $g$ on $M^*$,
and $\bdiv_g$ (resp.~$\bdiv_{g^{(h)}}$)
the distributional divergence with respect to $g$ (resp.~$g^{(h)}$).
Remark that to obtain this, we need some geometric property of
singularities of $M$ (see the proofs of Lemmas \ref{lem:gh} and
\ref{lem:divconv}) besides the BV property of $g$.

Using the Green formula, we prove the Laplacian Comparison
Theorem, \ref{thm:LapComp}.
Our idea is to approximate any region $E$ with piecewise smooth boundary
by the union of finitely many regions $A_k$,
where each $A_k$ forms the intersection
of some concentric annulus centered at $p$ of radii $r_k^- < r_k^+$
and a union of minimal geodesics emanating from $p$.
See Figure \ref{fig:E}.

\begin{figure}[htbp]
  \centering
  \input{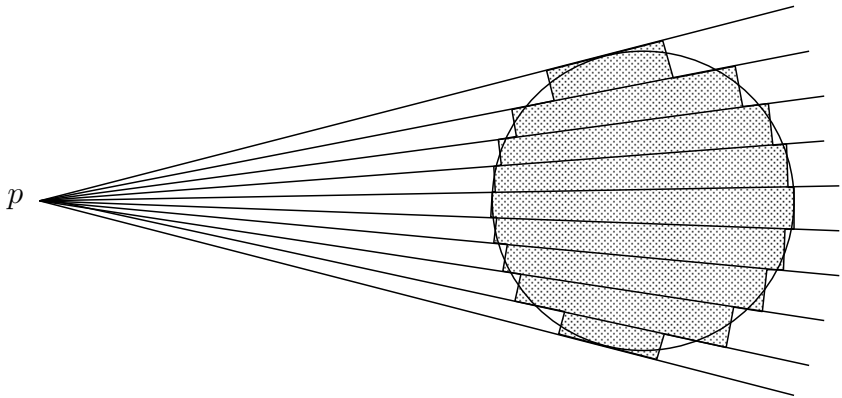}
  \caption{Approximate $E$ by $\bigcup_k A_k$.}
  \label{fig:E}
\end{figure}

Set $B_k^- := \partial B(p,r_k^-) \cap \partial A_k$
and $B_k^+ := \partial B(p,r_k^+) \cap \partial A_k$.
We assume that each $A_k$ is very thin, i.e., the diameters of $B_k^\pm$ are
very small.
We note that $\bigcup_k (B_k^- \cup B_k^+)$ approximates $\partial E$
and $\partial E$ has a division corresponding to $\{A_k\}$.
$B_k^\pm$ are all perpendicular to $\nabla r_p$ and
the area of $B_k^\pm$ is close to that of the corresponding part of
$\partial E$ multiplied by $\langle \nu_E,\nabla r_p \rangle$.
Since the cut-locus of $p$ could be very complex
(e.g. could be a dense subset), we need a delicate discussion.
Using $\BG(\kappa)$, we estimate
the difference between the areas of $B_k^-$ and $B_k^+$
by the volume of $A_k$.
Summing up this for all $k$, we have
an estimate of the right-hand side of the Green formula
for $E$ by the volume of $E$, that is,
\[
\blap r_p(E) \ge -(n-1) \sup_{x \in E} \cot_\kappa(r_p(x)) \;\Hm^n(E).
\]
This implies the Laplacian Comparison Theorem.

The organization of this paper is as follows.
In \S\ref{sec:prelim} we prepare Alexandrov spaces and BV functions.
In \S\ref{sec:BG} we prove some basic properties for
Condition $\BG(\kappa)$.
In \S\ref{sec:Green}, we perform some serious BV calculus on
Alexandrov spaces and prove the Green formula.
In \S\ref{sec:LapComp}, we give a proof of the Laplacian Comparison
Theorem, \ref{thm:LapComp}.
In the final section, \S\ref{sec:splitting}, we prove
Theorem \ref{thm:splitting} following the method of Cheeger-Gromoll
\cite{CG:split}.

\section{Preliminaries} \label{sec:prelim}

\subsection{Notation}  \label{ssec:theta}
Let $\theta(x)$ be some function of variable $x \in \R$ such that
$\theta(x) \to 0$ as $x \to 0$, and
$\theta(x|y_1,y_2,\dots)$ some function of variable $x \in \R$
depending on $y_1,y_2,\dots$ such that
$\theta(x|y_1,y_2,\dots) \to 0$ as $x \to 0$.
We use them like Landau's symbols.

\subsection{Alexandrov spaces and their structure} \label{ssec:Alex}

In this section, we present basics for Alexandrov spaces.
Refer \cite{BBI:course,BGP,OS:rstralex,Pr:DC} for the details.

Let $M$ be a \emph{geodesic space}, i.e., any two points
$p,q \in M$ can be joined by a length-minimizing curve,
called a \emph{minimal geodesic $pq$}.  Note that for given $p,q \in M$
a minimal geodesic $pq$ is not unique in general.
A \emph{triangle $\triangle pqr$ in $M$} means a set of three points
$p,q,r \in M$ (\emph{vertices}), and of three geodesics
$pq$, $qr$, $rp$ (\emph{edges}).
For a number $\kappa \in \R$, a \emph{$\kappa$-comparison triangle of
a triangle $\triangle pqr$ in $M$} is defined to be a triangle
$\triangle\tilde p \tilde q \tilde r$ in
a complete simply connected space form of curvature $\kappa$
with the property that
$d(p,q) = d(\tilde p,\tilde q)$, $d(q,r) = d(\tilde q,\tilde r)$,
$d(r,p) = d(\tilde r,\tilde p)$.
We denote by $\tilde\angle pqr$ the angle $\angle\tilde p \tilde q \tilde r$
between $\tilde q\tilde p$ and $\tilde q\tilde r$ at $\tilde q$
of $\triangle\tilde p \tilde q \tilde r$.
$\tilde\angle pqr$ is determined only by $d(p,q)$, $d(q,r)$, $d(r,p)$,
and $\kappa$.

\begin{defn}[Alexandrov Convexity]
  A subset $\Omega \subset M$ is said to satisfy
  the ($\kappa$-)\emph{Alexandrov convexity} if
  for any triangle $\triangle pqr \subset \Omega$,
  there exists a $\kappa$-comparison triangle
  $\triangle\tilde p\tilde q \tilde r$ such that
  for any $x \in pq$, $y \in pr$, $\tilde x \in \tilde p \tilde q$,
  $\tilde y \in \tilde p\tilde r$ with
  $d(p,x) = d(\tilde p,\tilde x)$, $d(p,y) = d(\tilde p,\tilde y)$
  we have
  \[
  d(x,y) \ge d(\tilde x,\tilde y).
  \]
\end{defn}


\begin{defn}[Lower Bound of Curvature, $\underline{\kappa}$]
  For a subset $\Omega \subset M$,
  we denote by $\underline{\kappa}(\Omega)$ the supremum
  of $\kappa \in \R$ for which $\Omega$ satisfies
  the $\kappa$-Alexandrov convexity.
  $\underline{\kappa}(\Omega)$ may be $+\infty$ or $-\infty$.
  For a point $x \in M$ we set
  $\underline{\kappa}(x) := \sup_U \underline{\kappa}(U)$,
  where $U$ runs over all neighborhoods of $x$.
\end{defn}

The function $\underline{\kappa} : M \to [\,-\infty,+\infty]$ is
lower semi-continuous.

\begin{defn}[Alexandrov Space]
  We say that $M$ is an \emph{Alexandrov space} if
  \begin{enumerate}
  \item $M$ is a complete geodesic space,
  \item $\underline{\kappa}(x) > -\infty$ for any $x \in M$,
  \item the Hausdorff dimension of $M$ is finite.
  \end{enumerate}
  An Alexandrov space $M$ is said to be of \emph{curvature $\ge \kappa$}
  if $\underline{\kappa}(M) \ge \kappa$.
  We usually assume the connectedness for Alexandrov spaces.
  However, we agree that a two-point space $M = \{p,q\}$
  is an Alexandrov space of curvature $\ge \pi^2/d(p,q)^2$.
\end{defn}

Let $M$ be an Alexandrov space.
Then, $M$ is \emph{proper}, i.e., any bounded subset is relatively compact.
If $M$ is of curvature $\ge \kappa > 0$, then $\diam M \le \pi/\sqrt{\kappa}$
and $M$ is compact.
By the globalization theorem, for any bounded subset $\Omega \subset M$,
there exists $R > 0$ such that
\[
\underline{\kappa}(\Omega)
\ge \inf_{x \in B(\Omega,R)} \underline{\kappa}(x) > -\infty,
\]
where $B(\Omega,R)$ is the $R$-neighborhood of $\Omega$.
In particular we have
\[
\underline{\kappa}(M) = \inf_{x \in M} \underline{\kappa}(x)\ (\ge -\infty).
\]

The Hausdorff dimension of (any open subset of) $M$
is a non-negative integer and
coincides with the covering dimension.
A zero-dimensional Alexandrov space is a one-point or two-point space.
A one-dimensional Alexandrov space is a one-dimensional complete
Riemannian manifold possibly with boundary.

Let $n$ be the dimension of $M$ and assume $n \ge 1$.
We take any point $p \in M$ and fix it.
Denote by $\Sigma_pM$ the space of directions at $p$,
and by $K_pM$ the tangent cone at $p$ (see \cite{BGP}).
$\Sigma_pM$ is an $(n-1)$-dimensional compact Alexandrov space
of curvature $\ge 1$ and
$K_pM$ an $n$-dimensional Alexandrov space of curvature $\ge 0$.
If $M$ is a Riemannian manifold, $\Sigma_pM$ and $K_pM$
are identified respectively
with the unit tangent sphere and the tangent space.


\begin{defn}[Singular Point, $\delta$-Singular Point]
  A point $p \in M$ is called a \emph{singular point of $M$}
  if $\Sigma_pM$ is not isometric to the unit sphere $S^{n-1}$.
  Let $\delta > 0$.
  We say that a point $p \in M$ is \emph{$\delta$-singular}
  if $\Hm^{n-1}(\Sigma_pM) \le \vol(S^{n-1})-\delta$.
  Let us denote the set of singular points of $M$ by
  $S_M$ and the set of $\delta$-singular points of $M$ by
  $S_\delta$.  
\end{defn}

We have $S_M = \bigcup_{\delta > 0} S_\delta$.
Since the map $M \ni p \mapsto \Hm^n(\Sigma_pM)$ is lower semi-continuous,
the $\delta$-singular set $S_\delta$ is a closed set and so the singular
set $S_M$ is a Borel set.
For a sufficiently small $\delta > 0$, any point in $M \setminus S_\delta$
has some Euclidean neighborhood.
For any geodesic segment $pq$ and any $x,y \in pq \setminus \{p,q\}$,
$\Sigma_xM$ and $\Sigma_yM$ are isometric to each other (\cite{Pt:parallel}).
Therefore, a geodesic joining two points in $M \setminus S_M$ is
entirely contained in $M \setminus S_M$.

\begin{defn}[Boundary]
  The boundary of an Alexandrov space $M$ is defined inductively.
  If $M$ is one-dimensional, then $M$ is a complete Riemannian manifold
  and the \emph{boundary of $M$} is defined as usual.
  Assume that $M$ has dimension $\ge 2$.
  A point $p \in M$ is a \emph{boundary point of $M$} if
  $\Sigma_pM$ has non-empty boundary.
\end{defn}

Any boundary point of $M$ is a singular point.
More strongly, the boundary of $M$ is contained in
$S_\delta$ for a sufficiently small $\delta > 0$, which follows from
the Morse theory in \cite{Pr:alex2,Pr:morse}.

The doubling theorem (\S 5 of \cite{Pr:alex2}; 13.2 of \cite{BGP})
states that if $M$ has non-empty boundary, then the \emph{double of $M$}
(i.e., the gluing of two copies of $M$ along their boundaries)
is an Alexandrov space without boundary and each copy of $M$ is
convex in the double.

Denote by $\hat S_M$ (resp.~$\hat S_\delta$)
the set of singular (resp.~$\delta$-singular) points of
$\dbl(M)$ contained in $M$, where we consider $M$ as a copy in
$\dbl(M)$.
We agree that $\hat S_M = S_M$ and $\hat S_\delta = S_\delta$ provided
$M$ has no boundary.

\begin{fact} \label{fact:str}
  For an Alexandrov space $M$ of dimension $n \ge 2$, we have the
  following {\rm(1)--(5)}.
  \begin{enumerate}
  \item  There exists a number $\delta_n > 0$ depending only on $n$
    such that
    $M^* := M \setminus \hat S_{\delta_n}$ is a manifold {\rm(}with
    boundary{\rm)}
    {\rm(\cite{BGP,Pr:alex2,Pr:morse})}
    and have a natural $C^\infty$ differentiable structure
    {\rm(}even on the boundary{\rm)}
    {\rm(\cite{KMS:lap})}.
  \item The Hausdorff dimension of $S_M$ is $\le n-1$
    {\rm(\cite{BGP, OS:rstralex})},
    and that of $\hat S_M$ is $\le n-2$ {\rm(\cite{BGP})}.
  \item We have a unique Riemannian metric $g$ on $M^* \setminus \hat S_M$
    such that the distance function induced from $g$ coincides with
    the original one of $M$ {\rm(\cite{OS:rstralex})}.
  \item For any $\delta$ with $0 < \delta \le \delta_n$,
    there exists a $C^\infty$ Riemannian metric
    $g_\delta$ on $M \setminus \hat S_\delta$ such that
    \[
    | g - g_\delta | < \theta(\delta|n)
    \quad\text{on $M^* \setminus \hat S_M$}
    \]
    {\rm(\cite{KMS:lap})}, where $\theta(\delta|n)$ is defined in
    \S\ref{ssec:theta}.
  \item A $C^\infty$ differentiable structure on
    $M^*$ satisfying {\rm(4)}
    is unique {\rm(\cite{KMS:lap})}.
    In this meaning, the $C^\infty$ structure is canonical.
  \end{enumerate}
\end{fact}

\begin{rem}
  In \cite{KMS:lap} we construct a $C^\infty$ structure only on
  $M \setminus B(S_{\delta_n},\epsilon)$.
  However this is independent of $\epsilon$ and extends to
  $M^*$.
  The $C^\infty$ structure is a refinement of the structures
  of \cite{OS:rstralex,Ot:secdiff,Pr:DC}.
  In particular, it is compatible with the DC structure of \cite{Pr:DC}.
\end{rem}

Note that the metric $g$ is defined only on $M^* \setminus \hat S_M$ and
does not continuously extend to any other point of $M$.
In general the non-singular set $M^* \setminus \hat S_M$ is not 
a manifold because $\hat S_M$ may be dense in $M$.

\begin{fact} \label{fact:g}
  $g$ is of locally bounded variation {\rm(\cite{Pr:DC};} see \S\ref{ssec:BV}
  below for functions of bounded variation{\rm)}.
  The tangent spaces at points in $M \setminus S_M$ is
  isometrically identified with the tangent
  cones {\rm(}\cite{OS:rstralex}{\rm)}.
  The volume measure on $M^*$ induced from $g$
  \[
  d\vol = d\vol_g := \sqrt{|g|} \; dx
  \]
  coincides with the $n$-dimensional Hausdorff measure $\Hm^n$
  {\rm(\cite{OS:rstralex})},
  where $dx := dx^1 \cdots dx^n$ is the Lebesgue measure on a chart.
  $g$ is \emph{uniformly elliptic} {\rm(\cite{OS:rstralex})}, i.e.,
  there exists a chart around each point in $M^*$
  on which
  \begin{itemize}
  \item[(UE)] the eigenvalues of $(g_{ij})$ are
    bounded away from zero and bounded from above.
  \end{itemize}
\end{fact}

We assume that all charts of $M^*$ satisfy (UE).

\begin{defn}[Cut-locus]
  Let $p \in M$ be a point.
  We say that a point $x \in M$ is a \emph{cut point of $p$}
  if no minimal geodesic $py$ from $p$ contains $x$ as an interior point.
  The set of cut points of $p$ is called the \emph{cut-locus of $p$} and
  denoted by $\Cut_p$.
\end{defn}

For the $W_{p,t}$ defined in \S\ref{sec:intro}, we have
$\bigcup_{0 < t < 1} W_{p,t} = X \setminus \Cut_p$.
Since $W_{p,t}$ is a closed set, the cut-locus $\Cut_p$ is a Borel set.
We have $\Hm^n(\Cut_p) = 0$ (Proposition 3.1 of \cite{OS:rstralex}).

By Lemma 4.1 of \cite{OS:rstralex},
$r_p := d(p,\cdot)$ is differentiable on
$M \setminus (S_M \cup \Cut_p \cup \{p\})$.
At any $x \in M \setminus (S_M \cup \Cut_p \cup \{p\})$
the gradient vector $\nabla r_p(x)$ coincides with
the tangent vector to the minimal geodesic from $p$ passing through $x$.
The gradient vector field $\nabla r_p$ is continuous at
all differentiable points.

\subsection{Analysis on Alexandrov spaces} \label{ssec:analysis}

Let $M$ be $n$-dimensional Alexandrov space and
$L^2(M)$ the Hilbert space consisting of all
real valued $L^2$ functions on $M$ with inner product
\[
(u,v)_{L^2} := \int_M uv\;d\Hm^n,
\qquad u,v \in L^2(M).
\]
We indicate the locally $L^2$ by $L^2_\loc$.
For a (uniformly elliptic) chart $(U;x^1,\dots,x^n)$ on $M^*$,
we denote by $D_i$ the distributional partial derivative
with respect to the coordinate $x^i$.
If $D_i u$ is a function for a function $u$,
we write it by $\partial_i u$.
Define $W^{1,2}(M)$ to be the set of all $u \in L^2(M)$
such that on each chart $(U;x^1,\dots,x^n)$ of $M^*$,
all $D_i u$, $i = 1,\dots,n$, are locally $L^2$ functions, $\partial_i u$, and
$\langle du,du\rangle := g^{ij}\,\partial_i u \,\partial_j u$
belongs to $L^1(M^*)$, where we follow Einstein's convention and
$\langle du,du\rangle$ is determined independent of the chart $U$.
$W^{1,2}_\loc(M)$ denotes the set of
$u \in L^2_\loc(M)$ such that all $D_iu$, $i = 1,\dots,n$, are
locally $L^2$ functions.
We define the symmetric bilinear form $\mathcal{E}$ on $W^{1,2}(M)$
by
\[
\mathcal{E}(u,v) := \int_{M^*} \langle du,dv \rangle \; d\Hm^n,
\qquad u,v \in W^{1,2}(M).
\]
We call $\mathcal{E}$ the \emph{Dirichlet energy form of $M$}.
$W^{1,2}(M)$ is a Hilbert space with inner product
$(u,v)_{L^2} + \mathcal{E}(u,v)$.
The pair $(\mathcal{E},W^{1,2}(M))$ becomes
a strongly local regular Dirichlet form in the sense of
 \cite{FOT} (Theorem 4.2 and Proposition 7.2 of \cite{KMS:lap}).
$\mathcal{E}(u,v)$ is defined for $u,v \in W^{1,2}_\loc(M)$ such that
$\supp v$ is compact in the same manner.

\begin{defn}[Sub-(super-)harmonicity] \label{defn:subh}
  A function $u \in W^{1,2}_\loc(M)$ is said to be
  \emph{$\mathcal{E}$-subharmonic} (resp.~\emph{$\mathcal{E}$-superharmonic})
  if for any $v \in C^\infty_0(M^*)$ with $v \ge 0$ we have
  $\mathcal{E}(u,v) \le 0$ (resp.~$\ge 0$).
\end{defn}

\begin{rem}
  Theorem 3.1 of \cite{KMS:lap} implies that
  $M \setminus M^*$ is an almost polar set in $M$.
  Therefore, the $\mathcal{E}$-sub(super)harmonicity defined here 
  is compatible with the terminology in \cite{Kw:maxprinsemi}.
\end{rem}

By Theorem 1.3 of \cite{Kw:maxprinsemi} and Theorem 3.1 of \cite{KMS:lap},
we have

\begin{lem}[Maximum Principle; \cite{Kw:maxprinsemi}] \label{lem:maxprin}
  Let $u \in W^{1,2}_\loc(M)$ be continuous and
  $\mathcal{E}$-subharmonic.
  If $u$ attains its maximum in $M$, then
  $u$ is constant on $M$.
\end{lem}




\subsection{BV functions} \label{ssec:BV}

We mention basics for BV functions needed in this paper.
For more details we refer to \cite{AFP}.

Let $U \subset \R^n$ be an open subset.

\begin{defn}[Approximate Limit] \label{defn:approxlim}
  We say that locally $L^1$ function $u : U \to \R$ has
  \emph{approximate limit at $x \in U$} if
  there exists $z \in \R$ such that
  \[
  \lim_{r \to 0} \frac{1}{|B(x,r)|} \int_{B(x,r)} |u(y) - z|\;dy = 0,
  \]
  where
  $B(x,r)$ is the Euclidean ball centered at $x$ of radius $r$
  and $|B(x,r)|$ its Lebesgue measure.
  Denote by $S_u$ the set of $x \in U$ where $u$ does not have
  approximate limit.
  For $x \in U \setminus S_u$, the above $z$ is unique and set
  $\tilde u(x) := z$.
  The function $\tilde u : U \setminus S_u \to \R$ is called the
  \emph{approximate limit of $u$}.
\end{defn}

$S_u$ is a Borel set and satisfies $\Hm^n(S_u) = 0$.
$\tilde u$ is a Borel function.

\begin{lem}[cf.~Proposition 3.64 of \cite{AFP}] \label{lem:approxlim}
  \begin{enumerate}
  \item For any bounded locally $L^1$ functions $u_1,u_2 : U \to \R$
    we have
    \begin{gather*}
      S_{u_1+u_2} \subset S_{u_1} \cup S_{u_2},
      \qquad
      S_{u_1u_2} \subset S_{u_1} \cup S_{u_2},\\
      \widetilde{u_1+u_2} = \tilde u_1 + \tilde u_2,
      \quad
      \widetilde{u_1u_2} = \tilde u_1 \tilde u_2
      \quad \text{on $U \setminus (S_{u_1} \cup S_{u_2})$}.
    \end{gather*}
  \item For any Lipschitz function $f : \R \to \R$ and
    locally $L^1$ function $u : U \to \R$ we have
    \[
    S_{f \circ u} \subset S_u,
    \qquad
    \widetilde{f \circ u} = f \circ \tilde u
    \quad \text{on $U \setminus S_u$}.
    \]
  \end{enumerate}
\end{lem}

\begin{defn}[Approximate Jump Point]
  For a locally $L^1$ function $u : U \to \R$,
  a point $x \in U$ is called an \emph{approximate jump point of $u$}
  if there exist $a,b \in \R$ and $\nu \in S^{n-1}$ such that
  $a \neq b$ and
  \begin{align*}
    \lim_{\rho \to 0} \frac{1}{|B^+_\nu(x,\rho)|}
    \int_{B^+_\nu(x,\rho)} | u(y)-a|\;dy &= 0,\\
    \lim_{\rho \to 0} \frac{1}{|B^-_\nu(x,\rho)|}
    \int_{B^-_\nu(x,\rho)} | u(y)-b|\;dy &= 0,
  \end{align*}
  where $B^+_\nu(x,\rho) := \{ y \in B(x,\rho) \mid \langle y-x,\nu \rangle
  > 0 \}$ and $B^-_\nu(x,\rho) := \{ y \in B(x,\rho) \mid
  \langle y-x,\nu \rangle < 0 \}$.
  Denote by $J_u$ the set of 
  approximate jump point of $u$.
\end{defn}

$J_u$ is a Borel set and satisfies $J_u \subset S_u$
(cf.~Proposition 3.69 of \cite{AFP}).

\begin{defn}[BV Function]
  An $L^1$ function $u : U \to \R$ is of \emph{BV}
  (\emph{bounded variation}) if
  the distributional derivatives $D_i u$, $i=1,\dots,n$,
  are all finite Radon measures.
  $|D_iu|$ denotes the total variation measure of $D_iu$.
\end{defn}

\begin{lem}[cf.~Lemma 3.76 of \cite{AFP}] \label{lem:Du0}
  Let $u : U \to \R$ be a BV function and
  $B \subset U$ a Borel set.
  \begin{enumerate}
  \item If $\Hm^{n-1}(B) = 0$, then $|D_iu|(B) = 0$.
  \item If $\Hm^{n-1}(B) < +\infty$ and $B \cap S_u = \emptyset$, then
    $|D_iu|(B) = 0$.
  \end{enumerate}
\end{lem}

\begin{lem}[Federer-Vol'pert; cf.~Theorem 3.78 of \cite{AFP}]
  For any BV function $u : U \to \R$ we have
  $\Hm^{n-1}(S_u \setminus J_u) = 0$.
\end{lem}

\begin{lem}[cf.~Theorem 3.96 of \cite{AFP}] \label{lem:chain}
  \begin{enumerate}
  \item For any BV functions $u_1, u_2 : U \to \R$ and $c_1, c_2 \in \R$,
    the linear combination $c_1 u_1 + c_2 u_2$ is also of BV
    and satisfies
    \[
    D_i(c_1 u_1 + c_2 u_2) = c_1 D_i u_1 + c_2 D_i u_2.
    \]
  \item Assume $|U| < \infty$.
    If $u : U \to \R$ is a BV function with $\inf u > 0$, $\sup u < +\infty$,
    and $\Hm^{n-1}(J_u) = 0$, and if
    $f : (\,0,+\infty\,) \to \R$ is a $C^1$ function, then
    $f \circ u$ is of BV and
    \[
    D_i(f \circ u) = (f' \circ \tilde u) D_i u.
    \]
  \end{enumerate}
\end{lem}

\begin{lem}[Leibniz Rule; cf.~Example 3.97 in \S 3.10 of \cite{AFP}]
  \label{lem:Leibniz}
  For any bounded BV functions $u_1, u_2 : U \to \R$ we have
  the following {\rm(1)} and {\rm(2)}.
  \begin{enumerate}
  \item $u_1u_2$ is of BV.
  \item If $\Hm^{n-1}(J_{u_1} \cap J_{u_2}) = 0$, then
    \[
    D_i(u_1u_2) = \tilde u_1 D_i u_2 + \tilde u_2 D u_1
    \]
    and $J_{u_1u_2}$ is contained in
    $(J_{u_1} \setminus S_{u_2}) \cup (J_{u_2} \setminus S_{u_1})$
    upto an $\Hm^{n-1}$-negligible set.
  \end{enumerate}
\end{lem}

\subsection{DC functions} \label{ssec:DC}

Let $\Omega$ be an open subset of an Alexandrov space.
($\Omega$ is allowed to be an open subset of $\R^n$.)
A function $u : \Omega \to \R$ is said to be \emph{convex}
if $u \circ \gamma$ is a convex function
for any geodesic $\gamma$ in $\Omega$.

\begin{defn}[DC Function]
  A locally Lipschitz function $u : \Omega \to \R$ is of \emph{DC}
  if it is locally represented as the difference of two convex functions,
  i.e., for any $p \in \Omega$ there exists two convex functions
  $v$ and $w$ on some neighborhood $U$ of $p$
  in $\Omega$ such that $u|_{U} = v - w$ on $U$.
\end{defn}

\begin{lem}[cf.~\S 6.3 of \cite{EG:meas}] \label{lem:derDC}
  For any DC function $u$ on $\Omega \subset \R^n$,
  the partial derivatives
  $\partial_i u$, $i = 1,2,\dots,n$, are all locally bounded and
  locally BV functions.
\end{lem}

\begin{lem}[Perelman; \S 3 of \cite{Pr:DC}] \label{lem:chartDC}
  Let $M$ be an $n$-dimensional Alexandrov space.
  For any chart $(U,\varphi)$ in $M^*$ and any DC function $u$ on $U$,
  $u \circ \varphi^{-1}$ is of DC on $\varphi(U) \subset \R^n$.
  In particular, $\partial_i u := \partial_i(u \circ \varphi^{-1})$,
  $i = 1,2,\dots,n$, are all locally bounded and locally BV functions.
\end{lem}

\begin{lem}[cf.~\cite{Pr:DC}] \label{lem:rpDC}
  For any point $p$ in an Alexandrov space $M$,
  $r_p(x) := d(p,x)$ is a DC function on $M \setminus \{p\}$.
\end{lem}

\section{Condition $\BG(\kappa)$} \label{sec:BG}

In this section, we mention some elementary and obvious properties of
Condition $\BG(\kappa)$.
Let $M$ be an $n$-dimensional Alexandrov space and
$p \in M$ a point.
For a subset $C \subset M$,
let $A_p(C) \subset M$ be the union of images of minimal geodesics
from $p$ intersecting $C$.
For $r > 0$ and $0 \le r_1 \le r_2$, we set
\begin{align*}
  a_C(r) &:= \Hm^{n-1}(A_p(C) \cap \partial B(p,r)),\\
  A_{r_1,r_2}(C) &:= \{ x \in A_p(C) \mid r_1 \le r_p(x) \le r_2\}.
\end{align*}

\begin{lem} \label{lem:artheta}
  Let $0 < r_1 \le r_2 \le R$ and $t := r_1/r_2$.
  Then we have
  \[
  a_C(r_1) \ge (1-\theta(t-1|R,\kappa_0)) \,a_C(r_2),
  \]
  where $\kappa_0$ is a lower bound of curvature on $B(p,2R)$,
  i.e., $\kappa_0 := \underline{\kappa}(B(p,2R))$, and
  $\theta(\cdots)$ is defined in \S\ref{ssec:theta}.
  In particular, if $a_C(r_0) = 0$ for a number $r_0 > 0$, then
  $a_C(r) = 0$ for any $r \ge r_0$.
  Moreover, we have $a_C(r+) \le a_C(r) \le a_C(r-)$ for any $r > 0$
  and $a_C$ has at most countably many discontinuity points.
\end{lem}

\begin{proof}
  To prove the first assertion, we
  assume that $A_p(C) \cap \partial B(p,r_2)$ is nonempty.
  The map $\Phi_{p,t} : A_p(C) \cap \partial B(p,r_1) \cap W_{p,t}
  \to A_p(C) \cap \partial B(p,r_2)$ defined in \S\ref{sec:intro}
  is surjective and Lipschitz continuous by the Alexandrov convexity.
  If $t$ is close to $1$, then
  so is the Lipschitz constant of $\Phi_{p,t}$.
  Therefore we have the first assertion of the lemma,
  which proves the rest.
\end{proof}

Lemma \ref{lem:artheta} implies the integrability of $a_C(r)$.
The same proof as for a Riemannian manifold leads to
\begin{equation}
  \label{eq:area}
  \Hm^n(A_{r_1,r_2}(C)) = \int_{r_1}^{r_2} a_C(r)\;dr
\end{equation}
for any $0 \le r_1 \le r_2$.

For a function $f : (\,\alpha,\beta\,) \to \R$, we set
\[
\bar f'(x) := \limsup_{h \to 0} \frac{f(x+h)-f(x)}{h}
\in \R \cup \{-\infty,+\infty\}, \quad \alpha < x < \beta.
\]

\begin{lem} \label{lem:ar}
  The following {\rm(1)--(3)} are equivalent to each other.
  \begin{enumerate}
  \item $\BG(\kappa)$ at $p$.
  \item For any subset $C \subset M$ and $0 < r_1 \le r_2$
    {\rm(}with $r_2 < \pi/\sqrt{\kappa}$ if $\kappa > 0${\rm)},
    we have
    \[
    a_C(r_1) \ge \frac{s_\kappa(r_1)^{n-1}}{s_\kappa(r_2)^{n-1}} a_C(r_2).
    \]
  \item For any $C \subset M$ and $r > 0$ with $a_C(r) > 0$
    {\rm(}and with $r < \pi/\sqrt{\kappa}$ if $\kappa > 0${\rm)}, we have
    \[
    \bar a_C'(r) \le (n-1)\cot_\kappa(r)\,a_C(r).
    \]
  \end{enumerate}
\end{lem}

\begin{proof}
  (1) $\implies$ (2):
  We fix $0 < r_1 \le r_2$.
  Assume that $r_2 < \pi/\sqrt{\kappa}$ if $\kappa > 0$.
  For a sufficiently small $\delta > 0$, we set
  $t_\delta := r_1/(r_2-\delta)$.
  Since $\Phi_{p,t_\delta}^{-1}(A_{r_2-\delta,r_2}(C)) \subset
  A_{r_1,t_\delta r_2}(C)$, we have by $\BG(\kappa)$ at $p$,
  \[
  \Hm^n(A_{r_1,t_\delta r_2}(C))
  \ge \min_{r_2-\delta \le r \le r_2}
  \frac{t_\delta s_\kappa(t_\delta r)^{n-1}}{s_\kappa(r)^{n-1}}
  \Hm^n(A_{r_2-\delta,r_2}(C)).
  \]
  We multiply the both sides of this formula by $1/(t_\delta \delta)$
  and take the limit as $\delta \to 0$.
  Then, remarking Lemma \ref{lem:artheta} we obtain (2).

  (2) $\implies$ (1):
  Let $E \subset M$ be a compact subset and set
  $r_E^- := \inf_{x \in E} r_p(x)$, $r_E^+ := \sup_{x \in E} r_p(x)$.
  We assume $r_E^+ < \pi/\sqrt{\kappa}$ if $\kappa > 0$.
  For $r \in [\,r_E^-,r_E^+\,]$, we set $C := \partial B(p,r) \cap E$,
  $r_2 := r$, and $r_1 := tr$.
  (2) implies
  \[
  \Hm^{n-1}(\Phi_{p,t}^{-1}(\partial B(p,r) \cap E)) \ge
  \frac{s_\kappa(tr)^{n-1}}{s_\kappa(r)^{n-1}}
  \Hm^{n-1}(\partial B(p,r) \cap E).
  \]
  Integrating this with respect to $r$ over $[\,r_E^-,r_E^+\,]$
  yields
  \[
  \Phi_{p,t\,*}\Hm^n(E) = \Hm^n(\Phi_{p,t}^{-1}(E))
  \ge \min_{r_E^- \le r \le r_E^+} 
  \frac{t \,s_\kappa(tr)^{n-1}}{s_\kappa(r)^{n-1}}
  \Hm^n(E),
  \]
  which implies $\BG(\kappa)$ at $p$.

  (2) $\Longleftrightarrow$ (3):
  We set $a(r) := a_C(r)$ for simplicity.
  Let $0 < r_1 \le r_2$ be any numbers such that
  $r_2 < \pi/\sqrt{\kappa}$ if $\kappa > 0$.
  In (2) we may assume that $r_1 < r_2$ and $a(r_1),a(r_2) > 0$,
  so that (2) is equivalent to
  \[
  \frac{\log a(r_2) - \log a(r_1)}{r_2 - r_1}
  \le (n-1) \frac{\log s_\kappa(r_2) - \log s_\kappa(r_1)}{r_2 - r_1}.
  \]
  This is also equivalent to that
  for any $r > 0$ with $a(r) > 0$ (and with $r < \pi/\sqrt{\kappa}$
  if $\kappa > 0$),
  \[
  \overline{(\log\circ a)}'(r) \le (n-1)\cot_\kappa(r).
  \]
  The left-hand side of this is equal to $\bar a'(r)/a(r)$.
\end{proof}

\begin{cor} \label{cor:rp}
  If $\kappa > 0$ and if $M$ satisfies $\BG(\kappa)$ at $p \in M$,
  then
  \[
  r_p \le \pi/\sqrt{\kappa}.
  \]
\end{cor}

\begin{proof}
  By Lemma \ref{lem:ar}(2), we have $a_M(r) \to 0$ as
  $r \to \pi/\sqrt{\kappa}$.
  By Lemma \ref{lem:artheta}, $a_M(r) = 0$ for any $r \ge \pi/\sqrt{\kappa}$.
  Using \eqref{eq:area} yields that
  $\Hm^n(M \setminus B(p,\pi/\sqrt{\kappa})) = 0$.
  This completes the proof.
\end{proof}

Lemma \ref{lem:ar}(2) leads to the following.

\begin{cor}[Bishop-Gromov Inequality]
  If $M$ satisfies $\BG(\kappa)$ at a point $p \in M$, then
  \[
  \frac{\Hm^n(B(p,r_1))}{\Hm^n(B(p,r_2))}
  \ge \frac{v_\kappa(r_1)}{v_\kappa(r_2)}
  \qquad\text{for $0 < r_1 \le r_2$},
  \]
  where $v_\kappa(r)$ is the volume of an $r$-ball in
  an $n$-dimensional complete simply connected space form of
  curvature $\kappa$.
\end{cor}

\begin{prop}[Stability for $\BG(\kappa)$]
  Let $M_i$, $i = 1,2,\dots$, and $M$ be $n$-dimensional compact
  Alexandrov spaces of curvature $\ge \kappa_0$
  for a constant $\kappa_0 \in \R$.
  If all $M_i$ satisfy $\BG(\kappa)$ and if 
  $M_i$ Gromov-Hausdorff converges to $M$, then
  $M$ satisfies $\BG(\kappa)$.
\end{prop}

\begin{proof}
  By \S 3 of \cite{Sy:raysalex} or 10.8 of \cite{BGP},
  $(M_i,\Hm^n)$ measured Gromov-Hausdorff converges to $(M,\Hm^n)$.
  The rest of the proof is omitted
  (cf.~\cite{CC:strRicI}).
\end{proof}

The following proposition and corollary are proved by
standard discussions
(cf.~Theorem 3.5 in Chapter IV of \cite{Sk:Riemgeom}).

\begin{prop} \label{prop:susp}
  Let $M$ be an Alexandrov space with $\BG(\kappa)$, $\kappa > 0$.
  Then we have $\diam M \le \pi/\sqrt{\kappa}$.
  If $\diam M = \pi/\sqrt{\kappa}$ then
  $M$ is homeomorphic to the suspension over some
  topological space.
\end{prop}



\begin{cor}
  Let $M$ be an Alexandrov space such that the singular set $S_M$
  is closed.  If $M \setminus S_M$ is an {\rm(}incomplete{\rm)} $C^\infty$
  Riemannian manifold of $\Ric \ge \kappa > 0$ and if
  $\diam M = \pi/\sqrt{\kappa}$, then $M$ is isometric to
  the spherical suspension over some compact Alexandrov space.
\end{cor}

The corollary is a generalization of Cheng's maximal
diameter theorem \cite{Ch:eigencomp}.
Compare also \cite{Bz:orbidiam}.

\section{Green Formula} \label{sec:Green}

Throughout this section, let $M$ be an $n$-dimensional Alexandrov space.
The purpose of this section is to prove the following Green formula,
which is needed for the proof of
the Laplacian Comparison Theorem, \ref{thm:LapComp}.

\begin{thm}[Green Formula] \label{thm:Green}
  Let $p \in M$ be a point and $E \subset M^* \setminus \{p\}$ a region
  satisfying Assumption \ref{asmp:E1} below and
  $\Hm^{n-1}(\Cut_p \cap \partial E) = 0$.
  Then, for any $C^\infty$ function $f : M^* \to \R$ we have
  \[
  \int_E f \, d\blap r_p = \int_E \langle\nabla f,\nabla r_p\rangle \; d\Hm^n
  + \int_{\partial E} f\,\langle \nu_E,\nabla r_p \rangle
  \; d\Hm^{n-1},
  \]
  where $\nu_E$ denotes the inward normal vector
  field along $\partial E$ of $E$.
  In particular,
  \[
  \blap r_p(E) = \int_{\partial E} \langle \nu_E,\nabla r_p \rangle
  \; d\Hm^{n-1}.
  \]
\end{thm}

\begin{asmp} \label{asmp:E1}
  $E$ is a compact region in $M^*$ with piecewise
  $C^\infty$ boundary such that
  $|D_k g_{ij}|(\partial E \cap U) = 0$
  for any $i,j,k = 1,2,\dots,n$ and for any chart $U$ of $M^*$.
\end{asmp}

Here, the \emph{piecewise $C^\infty$ boundary} means that
the boundary $\partial E$ is divided into two disjoint subsets
$\tilde\partial E$ and $\hat\partial E$ such that
$\tilde\partial E$ is an $(n-1)$-dimensional $C^\infty$ submanifold of $M^*$
and that $\hat\partial E$ is a closed set with
$\Hm^{n-1}(\hat\partial E) = 0$.

To define the distributional Laplacian $\blap$, let us consider
the distributional divergence of locally BV vector fields.
Let $\Omega$ be an open subset of $M^*$.
A \emph{locally BV vector field $X$ on $\Omega$} is defined as
a linear combination $X = X^i\partial_i$ 
of locally BV functions $X^1, \dots, X^n$ on each chart in $\Omega$ with
the compatibility condition under chart transformations.
For a locally $L^1$ function $u$ on $\Omega$,
the approximate jump set $J_u$ on a chart is defined.
It is easy to prove that $J_u$ is independent of the chart,
so that $J_u$ is defined as a subset of $\Omega$.
For a locally $L^1$ tensor $T$ on $M^*$,
the approximate jump set $J_T \subset \Omega$ is defined to be
the union of the approximate jump sets of all coefficients
of $T$.

\begin{defn}[Distributional Divergence]
  For a locally bounded and locally BV vector field $X$ on $\Omega$,
  the \emph{distributional divergence of $X$} is defined by
  \[
  \bdiv X = \bdiv_g X := D_i(\sqrt{|g|}\,X^i),
  \]
  where $X = X^i \partial_i$ on a chart.
  By Fact \ref{fact:g}, Lemmas \ref{lem:chain} and \ref{lem:Leibniz},
  $\sqrt{|g|}$ is a locally bounded and locally BV function.
  Since $\Hm^{n-1}(J_{\sqrt{|g|}}) \le \Hm^{n-1}(S_M \cap M^*) = 0$
  and by the Leibniz rule (Lemma \ref{lem:Leibniz}),
  $\bdiv X$ is determined independent of the local chart.
  $\bdiv X$ is a Radon measure on $\Omega$.
\end{defn}

\begin{rem}
  $\bdiv X$ is a generalization of $\divv X\, d\vol$ on
  a $C^\infty$ Riemannian manifold, where $\divv X$ is the usual
  divergence of a $C^\infty$ vector field $X$.
  $\bdiv X/\sqrt{|g|}$ is corresponding to $\divv X\,dx$
  and is not an invariant under chart transformations,
  where $dx$ is the Lebesgue measure on the chart.
\end{rem}

\begin{defn}[Distributional Laplacian]
  For a DC function $u$ on $\Omega$, the partial derivatives
  $\partial_i u$, $i=1,2,\dots,n$, are locally bounded and locally BV
  functions (see Lemmas \ref{lem:derDC} and \ref{lem:chartDC})
  and so the gradient vector field
  $\nabla u := g^{ij} \partial_j u \,\partial_i$ is a locally bounded
  and locally BV vector field on $\Omega$.
  $\nabla u$ is independent of the local chart.
  The \emph{distributional Laplacian of $u$}
  \[
  \blap u := -\bdiv \nabla u
  = -D_i(\sqrt{|g|} g^{ij} \partial_j u)
  \]
  is defined as a Radon measure on $\Omega$.
\end{defn}

$\blap u$ is corresponding to $\Delta u\, d\vol$,
where $\Delta$ is the usual Laplacian.
By Lemma \ref{lem:rpDC}, $\blap r_p$ is defined on $M^* \setminus \{p\}$.

\begin{lem} \label{lem:div}
  For any bounded BV vector field $X$ on $M^*$ with compact support in $M^*$,
  we have
  \[
  \int_{M^*} d\bdiv X = 0.
  \]
\end{lem}

\begin{proof}
  There is a finite covering $\{U_k\}$ of $\supp X$ consisting of
  charts of $M^*$ with compact closure $\bar U_k$.
  We take a $C^\infty$ partition of unity $\{\rho_k : M^* \to [\,0,1\,]\}_k$
  associated with the covering, i.e.,
  $\supp\rho_k \subset U_k$ and $\sum_k \rho_k = 1$ on $\supp X$.
  Since $X = \sum_k \rho_k X$ we have
  \[
  \int_{M^*} d\bdiv X
  = \sum_k \int_{U_k} dD_i(\sqrt{|g|} \rho_k X^i) = 0.
  \]
\end{proof}

For a locally $L^1$ vector field $X = X^i \partial_i$ on $\Omega$,
we define
\[
e(\nabla u,X) := \sqrt{|g|}\, \tilde X^i D_i u,
\]
where $\tilde X^i$ is the approximate limit of $X^i$
(see Definition \ref{defn:approxlim}).

\begin{lem} \label{lem:Green}
  Let $f : \Omega \to \R$ be a $C^\infty$ function,
  $u$ a bounded BV function with compact support in $\Omega$,
  and $v$ a DC function on $\Omega$.
  If $\Hm^{n-1}(J_u \cap J_{dv}) = 0$, then
  \begin{gather}
    \bdiv(f u \nabla v) = f\,e(\nabla u,\nabla v)
    + \tilde u \, \bdiv(f \nabla v),
    \tag{1} \label{eq:Green1}\\
    \int_\Omega f\,de(\nabla u,\nabla v)
    = -\int_\Omega \tilde u\; d\bdiv(f \nabla v). \tag{2} \label{eq:Green}
  \end{gather}
\end{lem}

\begin{proof}
  \eqref{eq:Green} is obtained by integrating
  \eqref{eq:Green1} on $\Omega$ and using Lemma \ref{lem:div}.

  We prove \eqref{eq:Green1}.
  We fix a relatively compact chart $(U;x^1,\dots,x^n)$
  with $\bar U \subset \Omega$.
  The uniform ellipticity of $(U;x^1,\dots,x^n)$ implies
  that $|U| < +\infty$.
  Since $g_{ij}$ are continuous on $M \setminus S_M$,
  we have $S_{g_{ij}} \subset S_M$ and
  $\tilde g_{ij} = g_{ij}$ on $M \setminus S_M$.
  By Lemma \ref{lem:approxlim},
  the same is true for $g^{ij}$ and $\sqrt{|g|}$.
  Since the Hausdorff dimension of $S_M \cap M^*$ is $\le n-2$
  and since $g_{ij}$, $g^{ij}$, and $\partial_j v$ are all BV functions
  on $U$,
  Lemmas \ref{lem:approxlim} and \ref{lem:Leibniz} show
  \begin{align*}
    \bdiv(f u \nabla v) &= D_i(\sqrt{|g|}\, u\, f\, g^{ij} \partial_j v)\\
    &= \sqrt{|g|}\, f\, g^{ij} \widetilde{\partial_j v} \, D_i u
    + \tilde u \, D_i(\sqrt{|g|} \, f\, g^{ij} \partial_j v)\\
    &= f\,e(\nabla u,\nabla v) + \tilde u \, \bdiv(f \nabla v).
  \end{align*}
\end{proof}

Let us study the $C^\infty$ mollifier $g^{(h)}$ of the Riemannian metric $g$
on $M^*$.
Let $\{U_\lambda\}$ be a locally finite covering of $M^*$ consisting
of relatively compact charts with $\bar U_\lambda \subset M^*$,
and $\{\rho_\lambda : M^* \to [\,0,1\,]\}$ an associated
partition of unity, i.e., 
each $\supp\rho_\lambda$ is a compact subset of $U_\lambda$ and
$\sum_\lambda \rho_\lambda = 1$.
Let $\eta \in C^\infty_0(\R^n)$ be such that
$\eta \ge 0$, $\eta(-x) = \eta(x)$,
$\supp\eta \subset B(o,1)$, and $\int_{\R^n} \eta dx = 1$.
We set $\eta_\epsilon := \epsilon^{-n}\eta(x/\epsilon)$, $\epsilon > 0$.
Denote by $g_{\lambda;ij}$ the coefficients of $g$ with respect to
the coordinate of $U_\lambda$.
For each $\lambda$, there exists $\epsilon_\lambda > 0$ such that
for any $\epsilon$ with $0 < \epsilon \le \epsilon_\lambda$,
$g_{\lambda;ij} * \eta_\epsilon(x) :=
\int_{U_\lambda} \eta_\epsilon(x-y) g_{\lambda;ij}(y)\;dy$
is a $C^\infty$ Riemannian metric on some neighborhood of
$\supp\rho_\lambda$.
For $0 < h \le 1$, $g^{(h)}_\lambda$ denotes
the metric tensor defined by $g_{\lambda;ij} * \eta_{\epsilon_\lambda h}(x)$.
Define the $C^\infty$ Riemannian metric $g^{(h)}$ on $M^*$ by
\[
g^{(h)} := \sum_\lambda \rho_\lambda g^{(h)}_\lambda.
\]
On each relatively compact chart $U$,
$g_{ij}^{(h)}$ is uniformly bounded.  As $h \to 0$,
$g_{ij}^{(h)} \to g_{ij}$ pointwise on $U \setminus S_M$ and
$D_k g_{ij}^{(h)} = \partial_k g_{ij}^{(h)}dx \to D_k g_{ij}$ weakly $*$
(cf.~Proposition 3.2 of \cite{AFP}).

\begin{lem}[Compare Fact \ref{fact:str}(4) (or Theorem 6.1 of \cite{KMS:lap})]
  \label{lem:gh}
  For any $\epsilon, \delta > 0$ with $\delta \le \delta_n$, we have
  \[
  \limsup_{h \to 0} \sup_{U \setminus (S_M \cup B(S_\delta,\epsilon))}
  |g^{(h)}_{ij} - g_{ij}| < \theta(\delta|U),
  \]
  where $\delta_n$ is that in Fact \ref{fact:str} and
  $\theta(\delta|U)$ depends also on the coordinates of $U$.
\end{lem}

\begin{proof}
  The same proof as of Lemma 3.2(1) of \cite{OS:rstralex} yields that
  for any $p,q \in M$ and $z \in U \setminus B(S_\delta,\epsilon)$,
  \[
  \sup_{x \in B(z,t)}
  |\angle pxq - \angle pzq| < \theta(\delta)+\theta(t|p,q,z).
  \]
  Hence, looking at the definition of $g$ in \cite{OS:rstralex},
  we have for any
  $z \in U \setminus B(S_\delta,\epsilon)$,
  \[
  \sup_{x,y \in U \cap B(z,t) \setminus S_M}
  |g_{ij}(x) - g_{ij}(y)| < \theta(\delta|U)+\theta(t|z,U).
  \]
  This and the relative compactness of $U \setminus B(S_\delta,\epsilon)$
  imply the lemma.
\end{proof}

\begin{lem} \label{lem:divconv}
  For any $C^\infty$ vector field $Y$ on $M^*$ we have
  \[
  \bdiv_{g^{(h)}} Y \to \bdiv_g Y \quad\text{weakly $*$}.
  \]
\end{lem}

\begin{proof}
  We take any relatively compact chart $(U;x^1,\dots,x^n)$ with
  $\bar U \subset M^*$ and fix it.
  Let $Y = Y^i \partial_i$.
  For $\hat g = g, g^{(h)}$ we have on $U$,
  \begin{align*}
      \bdiv_{\hat g} Y &= \frac{Y^i}{2\sqrt{|\hat g|}}
      \sum_{k=1}^n |(\hat g_{j1},\dots,\hat g_{j,k-1},D_i\hat g_{jk},
      \hat g_{j,k+1},\dots,\hat g_{jn})_{j=1,\dots,n}|\\
      &\quad + \partial_i Y^i \sqrt{|\hat g|}\,dx,
  \end{align*}
  which forms
  \[
  F^{ijk}(\hat g) D_i \hat g_{jk}
  + \partial_i Y^i \sqrt{|\hat g|}\,dx,
  \]
  where $F^{ijk}(g^{(h)})$ is a $C^\infty$ function and
  $F^{ijk}(g)$ is a BV function which is continuous on $U \setminus S_M$.

  We fix $i,j,k$ and set
  $f_h := F^{ijk}(g^{(h)})$,
  $f := F^{ijk}(g)$, $\mu_h := D_i g^{(h)}_{jk} = \partial_i g^{(h)}_{jk}\,dx$,
  and $\mu := D_i g_{jk}$.
  It suffices to prove that $f_h \mu_h \to f \mu$ weakly $*$.
  Denote the positive part of $\mu_h$ by $\mu_h^+$ and the negative part
  by $\mu_h^-$.
  There are a sequence $h_l \to 0$ and non-negative Radon measures
  $\nu^+$ and $\nu^-$ on $U$ such that $\mu_{h_l}^\pm \to \nu^\pm$ weakly $*$.
  It holds that $\mu = \nu^+ - \nu^-$.
  We do not know the positive (resp.~negative) part of $\mu$
  coincides with $\nu^+$ (resp.~$\nu^-$).
  For simplicity we write $h_l$ by $h$.
  Since $f_h\mu_h^\pm - f\nu^\pm = (f_h - f)\mu_h^\pm
  + f\mu_h^\pm - f\nu^\pm$, we have for any $\varphi \in C_0(U)$,
  \begin{equation}
    \label{eq:divconv1}
    \left| \int_U \varphi f_h \, d\mu_h^\pm
      - \int_U \varphi f \, d\nu^\pm \right|
    \le \int_U |\varphi| |f_h-f|\,d\mu_h^\pm
    + \left| \int_U \varphi f \, d\mu_h^\pm
      - \int_U \varphi f \, d\nu^\pm \right|.
  \end{equation}
  Take any $\epsilon, \delta > 0$ with $\delta \le \delta_n$.
  If $h \ll \delta,\epsilon$, then
  $|f_h-f| < \theta(\delta)$ on $U \setminus (S_M \cup B(S_\delta,\epsilon))$
  by Lemma \ref{lem:gh}.
  By the uniform boundedness of $g_{ij}^{(h)}$ and $g_{ij}$ on
  $U \setminus S_M$,
  we have $|f_h|, |f| \le c$ on $U \setminus S_M$, where $c$ is some constant
  independent of $h$.
  The limit-sup as $h \to 0$ of the first term of the right-hand side
  of \eqref{eq:divconv1} is
  \begin{align}
    \label{eq:divconv2}
    &\limsup_h \int_U |\varphi| |f_h-f|\,d\mu_h^\pm\\
    &\le \limsup_h \int_{U \cap B(S_\delta,\epsilon)} 2c|\varphi|\,d\mu_h^\pm
    + \limsup_h \int_{U \setminus B(S_\delta,\epsilon)}
    \theta(\delta)|\varphi|\,d\mu_h^\pm \notag \\
    &\le \int_{U \cap \overline{B(S_\delta,\epsilon)}} 2c|\varphi|\,d\nu^\pm
    +\theta(\delta)\int_U |\varphi|\,d\nu^\pm. \notag
  \end{align}
  To estimate the first term of the right-hand side, we prove:

  \begin{slem}
    We have $|\nu|(U \cap S_\delta) = 0$ for any $\delta > 0$,
    where $|\nu| := \nu^+ + \nu^-$.
  \end{slem}

  \begin{proof}
    By remarking the uniform ellipticity of the charts,
    a direct calculation shows that
    \[
    d|\mu_h| \le c' dx + c' \sum_{\lambda,l,m,a} \rho_\lambda\,
    d|D_{\lambda;a} g_{\lambda;lm} * \eta_{\epsilon_\lambda h}|,
    \]
    where $c'$ is some positive constant, $l,m,a$ run over all
    $1,2,\dots,n$, and $D_{\lambda;a}$ means
    $D_a$ for the coordinate of $U_\lambda$.
    According to Proposition 3.7 of \cite{AFP} we have, as $h \to 0$,
    $\rho_\lambda d|D_{\lambda;a} g_{\lambda;lm} * \eta_{\epsilon_\lambda h}|
    \to \rho_\lambda d|D_{\lambda;a} g_{\lambda;lm}|$ weakly $*$
    on $U_\lambda$, and hence
    \[
    d|\nu| \le c' dx + c' \sum_{\lambda,l,m,a} \rho_\lambda\,
    d|D_{\lambda;a} g_{\lambda;lm}|.
    \]
    Since the Hausdorff dimension of $S_\delta \cap M^*$ is $\le n-2$,
    and by Lemma \ref{lem:Du0}(1), this proves the sublemma.
  \end{proof}

  By the sublemma,
  taking $\delta \to 0$ after $\epsilon \to 0$ in \eqref{eq:divconv2},
  we have
  \[
  \lim_h \int_U |\varphi| |f_h-f|\,d\mu_h^\pm = 0.
  \]

  We are going to estimate the other term of \eqref{eq:divconv1}.
  There is a continuous function $\psi_{\delta,\epsilon} : U \to [\,0,1\,]$
  such that
  $\psi_{\delta,\epsilon} = 1$ on $U \cap B(S_\delta,\epsilon)$,
  $\psi_{\delta,\epsilon} = 0$ on $U \setminus B(S_\delta,2\epsilon)$.
  Set $\bar\psi_{\delta,\epsilon} := 1 - \psi_{\delta,\epsilon}$ and
  take a number $h_0$ with $0 < h_0 \ll \delta$.
  Since $\bar\psi_{\delta,\epsilon} \varphi f_{h_0}$ is continuous,
  we have
  $\lim_{h\to 0}\int_U \bar\psi_{\delta,\epsilon} \varphi f_{h_0} \; d\mu_h^\pm
  = \int_U \bar\psi_{\delta,\epsilon} \varphi f_{h_0} \; d\nu^\pm$.
  Moreover, by Lemma \ref{lem:gh},
  $|f_{h_0} - f| < \theta(\delta)$ on
  $U \setminus (S_M \cup B(S_\delta,\epsilon))$
  and therefore
  \begin{align}
    &\limsup_h \left| \int_U \bar\psi_{\delta,\epsilon} \varphi f \; d\mu_h^\pm
      - \int_U \bar\psi_{\delta,\epsilon} \varphi f \; d\nu^\pm \right|
    \label{eq:divconv-c1}\\
    &\le \limsup_h \left| \int_U \bar\psi_{\delta,\epsilon} \varphi f
      \; d\mu_h^\pm
      - \int_U \bar\psi_{\delta,\epsilon} \varphi f_{h_0} \; d\mu_h^\pm
    \right| \notag\\
    &\quad + \left|
      \int_U \bar\psi_{\delta,\epsilon} \varphi f_{h_0} \; d\nu^\pm
      -\int_U \bar\psi_{\delta,\epsilon} \varphi f \; d\nu^\pm \right|
    \notag\\
    &\le \theta(\delta) \limsup_h \int_U \bar\psi_{\delta,\epsilon} |\varphi|
    \;d\mu_h^\pm
    + \theta(\delta) \int_U \bar\psi_{\delta,\epsilon} |\varphi|
    \;d\nu^\pm \le \theta(\delta). \notag
  \end{align}
  We also have
  \begin{align}
    \limsup_h \left| \int_U \psi_{\delta,\epsilon} \varphi f \; d\mu_h^\pm
    \right|
    &\le \nu^\pm(B(S_\delta,3\epsilon)) \sup_U |\varphi f|
    \le \theta(\epsilon|\delta), \label{eq:divconv-c2} \\
    \left| \int_U \psi_{\delta,\epsilon} \varphi f \; d\nu^\pm \right|
    &\le \theta(\epsilon|\delta). \label{eq:divconv-c3}
  \end{align}
  Combining \eqref{eq:divconv-c1}, \eqref{eq:divconv-c2}, and
  \eqref{eq:divconv-c3} yields
  \begin{align*}
    &\limsup_h \left| \int_U \varphi f \; d\mu_h^\pm
      - \int_U \varphi f \; d\nu^\pm \right|\\
    &\le \limsup_h \left| \int_U \bar\psi_{\delta,\epsilon} \varphi f
      \; d\mu_h^\pm
      - \int_U \bar\psi_{\delta,\epsilon} \varphi f \; d\nu^\pm \right|\\
    &\quad + \limsup_h \left| \int_U \psi_{\delta,\epsilon} \varphi f
      \; d\mu_h^\pm \right|
    + \left| \int_U \psi_{\delta,\epsilon} \varphi f \; d\nu^\pm \right|\\
    &\le \theta(\delta) + \theta(\epsilon|\delta).
  \end{align*}
  Thus we obtain $f_h\mu_h^\pm \to f \nu^\pm$ and so
  $f_h \mu_h = f_h \mu_h^+ - f_h \mu_h^- \to f \nu^+ - f \nu^- = f \mu$.
  This completes the proof.
\end{proof}

We need Lemma \ref{lem:divconv} to prove:

\begin{lem} \label{lem:DIE}
  Let $E \subset M$ be a region satisfying Assumption \ref{asmp:E1}.
  Define $I_E(x) := 1$ for $x \in E$, $I_E(x) := 0$ for $x \in M \setminus E$.
  Then we have
  \begin{gather}
    D_i I_E = |g|^{-1/2} g_{ij} \nu_E^j
    \; \Hm^{n-1}\lfloor_{\partial E},
    \tag{1} \label{eq:DIE1}\\
    e(\nabla I_E,X)
    = \langle \nu_E,\tilde X \rangle \; \Hm^{n-1}\lfloor_{\partial E}
    \tag{2} \label{eq:DIE2}
  \end{gather}
  for any bounded measurable vector field $X$ on $M^*$,
  where $\nu_E = \nu_E^j \partial_j$ is the inward normal vector
  field along $\partial E$ of $E$ and $\lfloor$ indicates the restriction of
  a measure.
\end{lem}

\begin{proof}
  \eqref{eq:DIE1}:
  We take any $C^\infty$ vector field $Y$ on $M^*$.
  On the $C^\infty$ Riemannian manifold $(M^*,g^{(h)})$,
  the divergence formula implies
  \[
  \int_E \bdiv_{g^{(h)}} Y
  = - \int_{\partial E} \langle \nu_E^{(h)},Y \rangle_{g^{(h)}}
  \;d\vol_{(\partial E,g^{(h)})},
  \]
  where $\nu_E^{(h)}$ is the inward normal vector field
  along $\partial E$ with respect to the metric $g^{(h)}$.
  It follows from Assumption \ref{asmp:E1} that
  $|\bdiv_g Y|(\partial E) = 0$.
  Lemma \ref{lem:divconv}
  shows that the left-hand side of the above converges to
  $\int_E \bdiv_g Y$.
  Since $g^{(h)} \to g$ on $M^* \setminus S_M$,
  the right-hand side converges to
  $- \int_{\partial E} \langle \nu_E,Y \rangle_g\;
  d\Hm^{n-1}\lfloor_{\partial E}$.
  Therefore we have
  \[
  \int_E \bdiv_g Y = - \int_{\partial E}
  \langle \nu_E,Y \rangle_g\;d\Hm^{n-1}\lfloor_{\partial E},
  \]
  which implies \eqref{eq:DIE1}.

  \eqref{eq:DIE2} follows from \eqref{eq:DIE1} by a direct
  calculation.
\end{proof}

With the help of Lemma \ref{lem:DIE},
we finally prove the Green Formula.

\begin{proof}[Proof of Theorem \ref{thm:Green}]
  By Lemma \ref{lem:Green}(1), we have
  \[
  \bdiv(f \nabla r_p) = \langle\nabla f,\nabla r_p\rangle\;d\Hm^n
  - f \blap r_p,
  \]
  which implies
  \[
  \int_E d \bdiv(f \nabla r_p)
  = \int_E \langle\nabla f,\nabla r_p\rangle \; d\Hm^n
  - \int_E f \; d\blap r_p.
  \]
  By $J_{dr_p} \subset \Cut_p$, by the assumption for $E$, and by
  applying Lemmas \ref{lem:Green}(2), \ref{lem:DIE}(2),
  the left-hand side of the above is equal to
  \begin{align*}
    \int_{M^*} I_E \; d\bdiv(f \nabla r_p)
    &= -\int_{M^* \setminus \{p\}} f\; de(\nabla I_E,\nabla r_p)\\
    &= -\int_{\partial E} f\,\langle \nu_E,\widetilde{\nabla r_p}\rangle
    \; d\Hm^{n-1}.
  \end{align*}
  Since $\Hm^{n-1}(\Cut_p \cap \partial E) = 0$, we have
  $\widetilde{\nabla r_p} = \nabla r_p$ $\Hm^{n-1}$-a.e.~on $\partial E$.
  This completes the proof of the theorem.
\end{proof}

\section{Laplacian Comparison} \label{sec:LapComp}

We prove Theorem \ref{thm:LapComp} by using
the Green Formula (Theorem \ref{thm:Green}).
Let $a_C(r)$ and $A_{r_1,r_2}(C)$ be as defined in \S\ref{sec:BG}.

Lemma \ref{lem:ar}(3) implies the following.

\begin{lem} \label{lem:acomp}
  If $M$ satisfies $\BG(\kappa)$ at a point $p \in M$,
  then for any $C \subset M$ and $0 < r_1 \le r_2$,
  \[
  a_C(r_2) - a_C(r_1) \le (n-1)\cot_\kappa(r_1)
  \Hm^n(A_{r_1,r_2}(C)).
  \]
\end{lem}

Denote by $M^3(\kappa_0)$ the three-dimensional complete simply connected
space form of curvature $\kappa_0$.

\begin{fact}[Wald Convexity; \cite{Wd:convex, Br:distgeom}] \label{fact:Wald}
  Let $p_1,p_2,q_1,q_2 \in M$ be four points.  Take
  a sufficiently large domain $\Omega$ containing $p_1,p_2,q_1,q_2$ and set
  $\kappa_0 := \min\{\underline{\kappa}(\Omega),0\}$.
  Then there exist four points
  $\tilde p_1,\tilde p_2,\tilde q_1,\tilde q_2 \in M^3(\kappa_0)$
  and $i_0,j_0 = 1,2$ with $(i_0,j_0) \neq (1,2)$ such that
  \begin{gather*}
    d(p_1,p_2) = d(\tilde p_1,\tilde p_2),
    \quad d(q_1,q_2) \ge d(\tilde q_1,\tilde q_2),\\
    d(p_i,q_j) = d(\tilde p_i,\tilde q_j)
    \quad\text{for $(i,j) \neq (i_0,j_0)$},\\
    d(p_{i_0},q_{j_0}) \ge d(\tilde p_{i_0},\tilde q_{j_0}).
  \end{gather*}
  Moreover, for any $x_i \in p_iq_i$, $i=1,2$, if we take
  $\tilde x_i \in \tilde p_i\tilde q_i$ such that
  $d(p_i,x_i) : d(p_i,q_i)
  = d(\tilde p_i,\tilde x_i) : d(\tilde p_i,\tilde q_i)$, then
  we have
  \[
  d(x_1,x_2) \ge d(\tilde x_1,\tilde x_2).
  \]
\end{fact}

For $a,b \in \R$ (depending on a number $\delta > 0$),
we define $a \doteqdot b$ as $|a-b| < \theta(\delta)$.

\begin{fact}[5.6 of \cite{BGP}] \label{fact:dia}
  Take four points $p_1,p_2,q_1,q_2 \in M$ and set 
  $\kappa_0 := \min\{\underline{\kappa}(\Omega),0\}$ for a sufficiently large
  domain $\Omega$ containing $p_1,p_2,q_1,q_2$.
  If
  \[
  d(q_1,q_2) < \delta \min\{d(p_1,q_1),d(p_2,q_1)\}
  \quad\text{and}\quad
  \tilde\angle p_1q_1p_2 > \pi-\delta,
  \]
  then we have
  \[
  \tilde\angle p_1q_1q_2 + \tilde\angle p_2q_1q_2 \doteqdot \pi,
  \ \angle p_1q_1q_2 \doteqdot \tilde\angle p_1q_1q_2,
  \ \angle p_2q_1q_2 \doteqdot \tilde\angle p_2q_1q_2,
  \]
  where $\tilde\angle$ indicates the angle of a $\kappa_0$-comparison
  triangle.
\end{fact}

\begin{cor} \label{cor:dia}
  Under the same assumption as in Fact \ref{fact:dia},
  if we take a point $x \in p_1q_1$ such that
  $d(q_1,x) < \delta \min\{d(p_1,q_1),d(p_2,q_1)\}$, then
  \begin{align}
    \angle xq_1q_2 &\doteqdot \tilde\angle xq_1q_2, \tag{1}\\
    \angle q_1xq_2 &\doteqdot \tilde\angle q_1xq_2. \tag{2}
  \end{align}
\end{cor}

\begin{proof}
  (1): The Alexandrov convexity implies that
  \[
  \angle xq_1q_2 \ge \tilde\angle xq_1q_2 \ge \tilde\angle p_1q_1q_2
  \doteqdot \angle p_1q_1q_2 = \angle xq_1q_2.
  \]

  (2): The points $p_1,p_2,x,q_2$ satisfy the assumption
  of Fact \ref{fact:dia}.
  Take a point $q_1' \in p_2x$ with $d(x,q_1') = d(x,q_1)$ and
  use (1).  Then,
  \[
  \angle p_2xq_2 \doteqdot \tilde\angle q_1'xq_2.
  \]
  Since $\angle p_1 x p_2 \ge \tilde\angle p_1 x p_2 \doteqdot \pi$,
  we have $\tilde\angle q_1xq_1' \le \angle q_1xp_2
  = \pi - \angle p_1 x p_2 \doteqdot 0$.
  Therefore $\angle p_2xq_2 \doteqdot \angle q_1xq_2$ and
  $\tilde\angle q_1'xq_2 \doteqdot \tilde\angle q_1xq_2$,
  which imply (2).
\end{proof}

Let $E$ be a region satisfying the following.

\begin{asmp} \label{asmp:E}
  $E$ is a region in $M^* \setminus \{p\}$ satisfying
  Assumption \ref{asmp:E1} and $\Hm^{n-1}(\Cut_p \cap \partial E) = 0$.
  The smooth part of $\partial E$ is transversal to $\nabla r_p$.
\end{asmp}

Recall that $\partial E$ is divided into
the smooth part $\tilde\partial E$ and
the non-smooth part $\hat\partial E$.
By $\Hm^n(\Cut_p) = 0$, we have a lot of $E$'s satisfying
Assumption \ref{asmp:E}.

For $\rho > 0$ we set
\begin{align*}
  D &:= \tilde\partial E \setminus (S_M \cup \Cut_p
  \cup A_p(\hat\partial E \setminus \Cut_p)),\\
  D_\rho &:= \{\; x \in D \mid \text{there is $y \in M$ such that
    $x \in py$ and $d(x,y) \ge \rho$}\;\}.
\end{align*}
Namely, $x \in D$ if and only if the following (1) and (2) hold.
\begin{enumerate}
\item $x \in \tilde\partial E \setminus (S_M \cup \Cut_p)$.
\item If we extends $px$ to a minimal geodesic from $p$ hitting
  $\hat\partial E$, then it cannot be extended any more.
\end{enumerate}

It is obvious that $\bigcup_{\rho > 0} D_\rho = D$.

\begin{lem}
  \begin{enumerate}
  \item $D$ and $D_\rho$ are Borel subsets.
  \item We have $\Hm^{n-1}(\partial E \setminus D) = 0$.
  \end{enumerate}
\end{lem}

\begin{proof}
  (1): For $\rho > 0$, we set
  \[
  W_\rho := \{\; x \in M \mid \text{there is $y \in M$ such that
    $x \in py$ and $d(x,y) \ge \rho$}\;\}.
  \]
  Since $D_\rho = D \cap W_\rho$ and $W_\rho$ is closed,
  it suffices to prove that $D$ is a Borel set.
  In fact, $A_p(\hat\partial E \cap W_\rho)$ is closed,
  monotone non-increasing in $\rho$, and satisfies
  \begin{equation}
    \label{eq:AECut}
    \bigcup_{\rho > 0} A_p(\hat\partial E \cap W_\rho)
    = A_p(\hat\partial E \setminus \Cut_p),
  \end{equation}
  which is a Borel set.
  Since $\tilde\partial E$, $S_M$, $\Cut_p$ are all Borel,
  so is $D$.

  (2): We take any points $p_1,p_2 \in \tilde\partial E
  \cap A_p(\hat\partial E \cap W_\rho)$.
  For each $i = 1,2$,
  we extend $pp_i$ to a minimal geodesic from $p$ hitting
  $\hat\partial E \cap W_\rho$ and denote the hitting point by $x_i$.
  We further extends the geodesic beyond $x_i$ to the point, say $q_i$,
  such that $d(x_i,q_i) = \rho$.
  Such the points $x_i$ and $q_i$ necessarily exist because of
  $p_1,p_2 \in A_p(\hat\partial E \cap W_\rho)$.
  For the points $p_i,q_i,x_i$, $i=1,2$, we apply the Wald convexity
  (Fact \ref{fact:Wald}) and have
  $d(p_1,p_2) \le c \; d(x_1,x_2)$, where $c$ is a constant
  independent of $p_1$, $p_2$.
  Therefore,
  $\Hm^{n-1}(\tilde\partial E \cap A_p(\hat\partial E \cap W_\rho))
  \le c^{n-1} \Hm^{n-1}(\hat\partial E \cap W_\rho) = 0$,
  which together with \eqref{eq:AECut} implies
  $\Hm^{n-1}(\tilde\partial E \cap A_p(\hat\partial E \setminus \Cut_p)) = 0$.
  Combining this, $\Hm^{n-1}(S_M) = 0$, and
  $\Hm^{n-1}(\Cut_p \cap \partial E) = 0$, we obtain (2).
\end{proof}

For two points $x,y \in D_\rho$, we define the point $\pi_x(y)$
to be the intersection point of a minimal geodesic from $p$
passing through $y$ and $\partial B(p,r_p(x))$ (if any).
Since $\nabla r_p$ and $\tilde\partial E$ are transversal to each other,
if $d(x,y)$ is small enough compared with a given $x \in \tilde\partial E$,
such the intersection point $\pi_x(y)$ exists.

\begin{lem} \label{lem:A}
  For any subset $A \subset B(x,\delta) \cap D_\rho$ with $\Hm^{n-1}(A) > 0$,
  we have
  \[
  \left| \frac{|\langle \nu_E(x),\nabla r_p(x) \rangle|\, \Hm^{n-1}(A)}
    {\Hm^{n-1}(\pi_x(A))} - 1 \right| < \theta(\delta|x,\rho).
  \]
\end{lem}

\begin{proof}
  We fix $x$ and $\rho$, then we write
  $\theta(\delta) = \theta(\delta|x,\rho)$.
  Assume $\delta \ll \rho$.
  For $a,b \in \R$, we define
  $a \simeq b$ as $|a - b| \le \theta(\delta)|a|$.

  Let $y,z \in B(x,\delta) \cap D_\rho$ be two different points.
  We take a minimal geodesic, say $\sigma$ (resp.~$\tau$),
  from $p$ containing $py$ (resp.~$pz$) which has
  maximal length.
  It follows that $L(\sigma), L(\tau) \ge r-\delta+\rho \ge r+\rho/2$,
  where we set $r := r_p(x)$.

  \begin{slem} \label{slem:A} We have
    $d(\sigma(t_1),\tau(t_1)) \simeq d(\sigma(t_2),\tau(t_2))$
    for any $t_1$ and $t_2$ with $\sigma(t_i),\tau(t_i) \in B(x,\delta)$.
  \end{slem}
  \begin{proof}
    The Alexandrov convexity implies
    \[
    d(\sigma(t_1),\tau(t_1)) \ge (1-\theta(\delta)) d(\sigma(t_2),\tau(t_2)).
    \]

    An inverse estimate follows from applying
    the Wald convexity (Fact \ref{fact:Wald}) to 
    $p_1 := \sigma(t_1)$, $p_2 := \tau(t_1)$,
    $q_1 := \sigma(r+\rho/2)$, $q_2 := \tau(r+\rho/2)$,
    $x_1 := \sigma(t_2)$, $x_2 := \tau(t_2)$.
  \end{proof}

  Setting $y' := \tau(r_p(y))$ and $z' := \sigma(r_p(z))$ we have,
  by Sublemma \ref{slem:A},
  \[
  d(\pi_x(y),\pi_x(z)) \simeq d(y,y') \simeq d(z,z').
  \]
  Let $\alpha := \angle zyz'$.
  By Corollary \ref{cor:dia},
  $\alpha \doteqdot \tilde\angle zyz'$ and hence
  \[
  d(y,z)\sin\alpha \simeq d(\pi_x(y),\pi_x(z)).
  \]
  We also have
  \[
  |r_p(y) - r_p(z)| = d(y,z') \simeq d(y,z)\cos\alpha.
  \]
  We assume that $\delta$ is small enough compared with $x$.
  Then, there is a chart $(U,\varphi)$ of $M^*$ containing
  $B(x,\delta)$ such that $\varphi(\partial E)$ is a hyper-plane
  in $\varphi(U) \subset \R^n$ and $g_{ij}(x) = \delta_{ij}$.
  Let $c$ be a curve from $y$ to $z$ such that
  $\varphi\circ c$ is a Euclidean line segment in $\varphi(U)$.
  Since $\langle\dot c(s),\nabla r_p(c(s))\rangle \doteqdot
  \langle\dot c(0),\nabla r_p(y)\rangle$ and $L(c) \simeq d(y,z)$,
  the first variation formula leads to
  \[
  r_p(y) - r_p(z) \simeq d(y,z) \langle\dot c(0),\nabla r_p(y)\rangle.
  \]
  and so $\cos\alpha \doteqdot |\langle\dot c(0),\nabla r_p(y)\rangle|$.
  Therefore,
  \[
  d(y,z) \sqrt{1-\langle\dot c(0),\nabla r_p(y) \rangle^2}
  \simeq d(\pi_x(y),\pi_x(z)).
  \]
  Take a hyper-plane $H \subset \R^n$ containing $\varphi(x)$
  and perpendicular to $\nabla r_p(x)$.
  Denoting the orthogonal projection by $P :
  \varphi(\partial E) \to H$,
  we see
  \begin{align*}
      d_{\R^n}(P(\varphi(y)),P(\varphi(z))) &=
      d(y,z) \sqrt{1-\langle\dot c(0),\nabla r_p(y) \rangle^2}\\
      &\simeq d(\pi_x(y),\pi_y(z)),
  \end{align*}
  which implies that $\Hm^{n-1}(\pi_x(A)) \simeq \Hm^n(P(\varphi(A)))$.
  Since $g_{ij} \doteqdot \delta_{ij}$ on $U$,
  we have $d(y,z) \simeq d_{\R^n}(\varphi(y),\varphi(z))$ and
  $\Hm^{n-1}(A) \simeq \Hm^{n-1}(\varphi(A))$.
  This completes the proof.
\end{proof}

\begin{proof}[Proof of Theorem \ref{thm:LapComp}]
  By the Green Formula (Theorem \ref{thm:Green}),
  it suffices to prove the theorem that
  \begin{equation}
    \label{eq:E}
    \int_{\partial E} \langle \nu_E,\nabla r_p \rangle \; d\Hm^{n-1}
    \ge -(n-1)\sup_{x \in E} \cot_\kappa(r_p(x))
    \Hm^n(E)
  \end{equation}
  for any region $E$ satisfying Assumption \ref{asmp:E}.

  We define
  \begin{align*}
    D^- &:= \{\; x \in D \mid \langle \nu_E(x),\nabla r_p(x) \rangle < 0\;\},\\
    D^+ &:= \{\; x \in D \mid \langle \nu_E(x),\nabla r_p(x) \rangle > 0\;\},\\
    D_\rho^\pm &:= D_\rho \cap D^\pm.
  \end{align*}
  They are all $\Hm^{n-1}$-measurable sets.  Take any
  $\epsilon > 0$ and fix it for a moment.
  Let $\delta_{x,\rho} > 0$ be a number small enough
  compared with $x$,$\rho$, and $\epsilon$.
  For the $\theta(\delta|x,\rho)$ of Lemma \ref{lem:A},
  we assume $\theta(\delta_{x,\rho}|x,\rho) \le \epsilon$.
  For a point $x \in \partial D^-$,
  let $\pi(x)$ be the intersection point of $px$ and $\partial E$ 
  which is nearest to $x$.
  From the definition of $D$ we have $\pi(D^-) \subset D^+$.
  Take a countable dense subset $\{x^-_k\}_k \subset D_\rho^-$
  and set $x^+_k := \pi(x^-_k)$.
  It holds that $x^+_k \in D^+$.
  We find a number $\delta_k$ in such a way that
  $0 < \delta_k < \delta_{x^-_k,\rho}$ and
  $\pi(B(x^-_k,\delta_k) \cap D^-_\rho)
  \subset B(x^+_k,\delta_{x^+_k,\rho})$.
  It follows from
  $D_\rho^- \subset \bigcup_k B(x^-_k,\delta_k)$ that
  there are disjoint $\Hm^{n-1}$-measurable subsets
  $B^-_k \subset B(x^-_k,\delta_k)$
  with $D_\rho^- = \bigcup_k B^-_k$.
  Setting $B^+_k := \pi(B^-_k)$ we have $B^+_k \subset D^+_\rho$.
  The definition of $\delta_k$ and Lemma \ref{lem:A} lead to
  \[
  \left| \frac{|\langle \nu_E(x^\pm_k),\nabla r_p(x^\pm_k) \rangle|
      \,\Hm^{n-1}(B^\pm_k)}
    {\Hm^{n-1}(\pi_{x^\pm_k}(B^\pm_{x^\pm_k}))} - 1 \right| < \epsilon.
  \]
  Taking $\delta_k$ small enough, we assume that
  \[
  | \langle \nu_E(x^\pm_k),\nabla r_p(x^\pm_k)\rangle
  - \langle \nu_E,\nabla r_p \rangle | < \epsilon
  \quad\text{on $B^\pm_k$}.
  \]
  Thus we have
  \begin{align*}
    &\int_{D^-_\rho \cup \pi(D^-_\rho)} \langle \nu_E,\nabla r_p \rangle
    \; d\Hm^{n-1}\\
    &= \sum_k \left\{
      \int_{B^-_k} \langle \nu_E,\nabla r_p \rangle \; d\Hm^{n-1}
      + \int_{B^+_k} \langle \nu_E,\nabla r_p \rangle \; d\Hm^{n-1}
    \right\}\\
    &\ge \sum_k \left\{
      \langle \nu_E(x^-_k),\nabla r_p(x^-_k) \rangle \Hm^{n-1}(B^-_k)
      +\langle \nu_E(x^+_k),\nabla r_p(x^+_k) \rangle \Hm^{n-1}(B^+_k)
    \right\}\\
    &\quad - 2\epsilon \Hm^{n-1}(\partial E)\\
    &\ge \sum_k \left\{
      \Hm^{n-1}(\pi_{x^+_k}(B^+_k)) - \Hm^{n-1}(\pi_{x^-_k}(B^-_k))
    \right\} - 4\epsilon \Hm^{n-1}(\partial E).\\
    \intertext{By Lemma \ref{lem:acomp}, this is}
    &\ge -(n-1) \left( \sup_E \cot_\kappa \circ r_p \right)
    \sum_k \Hm^n(A_{r^+_k,r^-_k}(\pi_{x^-_k}(B^-_k)))
    - 4\epsilon \Hm^{n-1}(\partial E).
  \end{align*}
  where we set $r^\pm_k := r_p(x^\pm_k)$.
  Let $A(D^-_\rho)$ be the region in $A_p(D^-_\rho)$
  between $D^-_\rho$ and $\pi(D^-_\rho)$.
  We assume the division $\{B^-_k\}_k$ of $D^-_\rho$ to be so fine that
  \[
  \left|\, \sum_k \Hm^n(A_{r^+_k,r^-_k}(\pi_{x^-_k}(B^-_k)))
  -\Hm^n(A(D^-_\rho)) \,\right| < \epsilon.
  \]
  Therefore,
  \begin{align*}
    &\int_{D^-_\rho \cup \pi(D^-_\rho)} \langle \nu_E,\nabla r_p \rangle
    \; d\Hm^{n-1}\\
    &\ge -(n-1) \left( \sup_E \cot_\kappa \circ r_p \right)
    (\Hm^n(A(D^-_\rho)) + \bar\epsilon)
    - 4\epsilon \Hm^{n-1}(\partial E),
  \end{align*}
  where $\bar\epsilon$ is either $\epsilon$ or $-\epsilon$.
  We define $A(D^-)$ as in the same manner as
  $A(D^-_\rho)$.
  After $\epsilon \to 0$ we take $\rho \to 0$ and then have
  \begin{equation} \label{eq:E1}
    \int_{D^- \cup \pi(D^-)} \langle \nu_E,\nabla r_p \rangle
    \; d\Hm^{n-1} 
    \ge -(n-1) \left( \sup_E \cot_\kappa \circ r_p \right)
    \Hm^n(A(D^-)).
  \end{equation}
  Set $D' := D^+ \setminus \pi(D^-)$.
  The set of $x \in \partial E$ such that
  $px$ passes through $D'$ is of $\Hm^{n-1}$-measure zero.
  Therefore, the same discussion as above leads to
  \begin{equation} \label{eq:E2}
    \int_{D'}  \langle \nu_E,\nabla r_p \rangle
    \; d\Hm^{n-1}  
    \ge -(n-1) \left( \sup_E \cot_\kappa \circ r_p \right)
    \Hm^n(A(D')),
  \end{equation}
  where $A(D')$ is the intersection of $E$ and
  the union of images of minimal geodesics from $p$
  intersecting $D'$.
  By \eqref{eq:E1} and \eqref{eq:E2} we obtain \eqref{eq:E}.
  This completes the proof.
\end{proof}

By using Theorem \ref{thm:LapComp}, a direct calculation implies

\begin{cor}
  Under the same assumption as in Theorem \ref{thm:LapComp},
  for any $C^2$ function $f : \R \to \R$ with $f' \ge 0$, we have
  \[
  \bar\Delta f\circ r_p
  \ge -\frac{(s_\kappa^{n-1} f')'}{s_\kappa^{n-1}}\circ r_p \;d\Hm^n
  \qquad\text{on $M^* \setminus \{p\}$.}
  \]
\end{cor}

\begin{rem} \label{rem:LapCut}
  $\blap r_p$ is not absolutely continuous with respect to $\Hm^n$
  on the cut-locus of $p$.
  In fact, let $M$ be an $n$-dimensional complete Riemannian manifold
  without boundary and $N \subset M$ a $k$-dimensional submanifold
  without boundary which is contained in $\Cut_p$ for a point $p \in M$.
  (We do assume the completeness of $N$.)
  Denote by $\nu(N)$ the normal bundle over $N$ and by $\nu_\epsilon(N)$
  the set of vectors in $\nu(N)$ with length $\le \epsilon$.
  We assume that there exists a number $\epsilon_0 > 0$ such that
  $\exp(\nu_{\epsilon_0}(N)) \cap \Cut_p = N$.
  For $x \in N$, let $V_x$ be the set of unit vectors at $x$ tangent to
  minimal geodesics from $x$ to $p$.
  $V_x$ is isometric to a $(n-k-1)$-sphere of radius $\in (\,0,1\,]$.
  The angle between $u$ and $V_x$,
  $\alpha(x) := \inf_{v \in V_x} \angle(u,v)$, is constant
  for all $u \in \nu_1(N) \cap T_xM$.
  Applying Theorem \ref{thm:Green} to
  $\nu_\epsilon(N')$, $N' \subset N$, $0 < \epsilon \le \epsilon_0$,
  we have
  \[
  d\blap r_p\lfloor_N(x)
  = \omega_{n-k-1} \cos\alpha(x) \; d\Hm^k\lfloor_N(x),
  \]
  where $\omega_{n-k-1}$ is the volume of a unit $(n-k-1)$-sphere.
  In particular, $\blap r_p$ is not absolutely continuous with respect to
  $\Hm^n$ on $N$.
\end{rem}

The following is needed in the proofs of Theorem \ref{thm:splitting},
Corollaries \ref{cor:heatcomp} and \ref{cor:eigen}.

\begin{cor} \label{cor:LapCompform}
  For any $f \in C_0^\infty(M^* \setminus \{p\})$ with $f \ge 0$, we have
  \[
  \int_{M^*} \langle \nabla r_p,\nabla f \rangle \; d\Hm^n
  \ge \int_{M^*} f\; d\blap r_p
  \ge -(n-1) \int_{M^*} f \, \cot_\kappa \circ r_p \; d\Hm^n.
  \]
\end{cor}

\begin{proof}
  Theorem \ref{thm:LapComp} says the second inequality of
  the corollary.
  In the case where $M^*$ has no boundary, the Green Formula
  (Theorem \ref{thm:Green}) tells us that the first term is
  equal to the second.
  We prove the first inequality in the case where $M^*$ has
  non-empty boundary.
  Assume that $M^*$ has non-empty boundary.
  Since $M^*$ is a $C^\infty$ manifold with $C^\infty$ boundary
  $\partial M^*$, we can approximate
  $\partial M^*$ by a $C^\infty$ hypersurface $N \subset M^*$
  with respect to the $C^1$ topology such that $\Hm^{n-1}(\Cut_p \cap N) = 0$
  and $|D_k g_{ij}|(N \cap U) = 0$
  for any $i,j,k$ and for any chart $U$ of $M^*$.  Let $V_N$ be
  the closed region in $M^*$ bounded by $N$ and
  not containing the boundary of $M^*$.
  We find a compact region $E \subset M^* \setminus \{p\}$
  satisfying the assumption of the Green Formula
  (Theorem \ref{thm:Green}) such that
  $V_N \cap \supp f \subset E \subset V_N$.
  The Green Formula implies
  \[
  \int_{V_N} \langle \nabla r_p,\nabla f \rangle \; d\Hm^n
  = \int_{V_N} f \; d\blap r_p
  - \int_{N \cap \supp f} f \, \langle \nu_N,\nabla r_p \rangle \; d\Hm^{n-1},
  \]
  where $\nu_N$ is the inward unit normal vector fields on $N$
  with respect to $V_N$.
  Since $M$ is convex in the double of $M$,
  as $N$ converges to the boundary of $M^*$ in the $C^1$ topology,
  any limit of $\langle \nu_N,\nabla r_p \rangle$ is non-positive,
  which together with Fatou's lemma shows
  \[
  \limsup_{N \to \partial M^*}
  \int_{N \cap \supp f} f \, \langle \nu_N,\nabla r_p \rangle
  \; d\Hm^{n-1} \le 0.
  \]
  This completes the proof.
\end{proof}

\section{Splitting Theorem} \label{sec:splitting}

We prove the Topological Splitting Theorem, \ref{thm:splitting},
following the idea of Cheeger-Gromoll \cite{CG:split}.

Let $M$ be a non-compact Alexandrov space and
$\gamma$ a \emph{ray} in $M$, i.e., a geodesic defined on $[\,0,+\infty\,)$
such that $d(\gamma(s),\gamma(t)) = |s-t|$ for any $s,t \ge 0$.

\begin{defn}[Busemann Function]
  The \emph{Busemann function $b_\gamma : M \to \R$ for $\gamma$}
  is defined by
  \[
  b_\gamma(x) := \lim_{t \to +\infty} \{ t - d(x,\gamma(t)) \},
  \quad x \in M.
  \]
\end{defn}

It follows from the triangle inequality that
$t - d(x,\gamma(t))$ is monotone non-decreasing in $t$,
so that the limit above exists.
$b_\gamma$ is a $1$-Lipschitz function.

\begin{defn}
  We say that a ray $\sigma$ in $M$ is \emph{asymptotic to $\gamma$}
  if there exist a sequence $t_i \to +\infty$, $i = 1,2,\dots$,
  and minimal geodesics
  $\sigma_i : [\,0,l_i\,] \to M$ with $\sigma_i(l_i) = \gamma(t_i)$
  such that $\sigma_i$ converges to $\sigma$ as $i \to \infty$,
  (i.e., $\sigma_i(t) \to \sigma(t)$ for each $t$).
\end{defn}

For any point in $M$, there is a ray asymptotic to $\gamma$
from the point.
Any subray of a ray asymptotic to $\gamma$ is asymptotic to $\gamma$.
By the same proof as for Riemannian manifolds
(cf.~Theorem 3.8.2(3) of \cite{SST:totcurv}), for any ray $\sigma$ asymptotic
to $\gamma$ we have
\begin{equation}
  \label{eq:blinear}
  b_\gamma \circ \sigma(s) = s + b_\gamma \circ \sigma(0)
  \quad\text{for any $s \ge 0$.}
\end{equation}
For a complete Riemannian manifold, $b_\gamma$ is differentiable at
$\sigma(s)$ for any $s > 0$, which seems to be true also for
Alexandrov spaces, but we do not need it for the proof of
Theorem \ref{thm:splitting}.

\begin{lem} \label{lem:diff}
  Let $f : M \to \R$ be a $1$-Lipschitz function and
  $u,v \in \Sigma_pM$ two directions at a point $p \in M$.
  If the directional derivative of $f$ to $u$ is equal to $1$
  and that to $v$ equal to $-1$,
  then the angle between $u$ and $v$ is equal to $\pi$.
\end{lem}

\begin{proof}
  There are points $x_t, y_t \in M$, $t > 0$, such that
  $d(p,x_t) = d(p,y_t) = t$ for all $t > 0$ and that
  the direction at $p$ of $px_t$ (resp.~$py_t$)
  converges to $u$ (resp.~$v$) as $t \to 0$.
  The assumption for $f$ tells us that
  \[
  \lim_{t \to 0} \frac{f(x_t) - f(p)}{t} = 1 \quad\text{and}\quad
  \lim_{t \to 0} \frac{f(y_t) - f(p)}{t} = -1,
  \]
  which imply
  \[
  \lim_{t \to 0} \frac{d(x_t,y_t)}t \ge
  \lim_{t \to 0} \frac{f(x_t) - f(y_t)}{t} = 2. 
  \]
  This completes the proof.
\end{proof}

\begin{lem} \label{lem:bfunc}
  Assume that a ray $\sigma : [\,0,\,+\infty\,) \to M$ is asymptotic
  to a ray $\gamma : [\,0,\,+\infty\,) \to M$, and let $s$ be a given
  positive number.
  Then, among all rays emanating from $\sigma(s)$,
  only the subray $\sigma|_{[\,s,+\infty\,)}$ of $\sigma$
  is asymptotic to $\gamma$.
\end{lem}

\begin{proof}
  Look at \eqref{eq:blinear} and use Lemma \ref{lem:diff} for
  $f := b_\gamma$.
\end{proof}

\begin{lem} \label{lem:bsplitting}
  Let $\gamma$ be a straight line in $M$.
  Denote by $b_+$ the Busemann function for
  $\gamma_+ := \gamma|_{[\,0,+\infty\,)}$ and by $b_-$
  that for $\gamma_- := \gamma|_{(\,-\infty,0\,]}$.
  If $b_+ + b_- \equiv 0$ holds, then
  $M$ is covered by disjoint straight lines bi-asymptotic to $\gamma$.
  In particular, $b_+^{-1}(t)$ for all $t \in \R$ are homeomorphic
  to each other and $M$ is homeomorphic to $b_+^{-1}(t) \times \R$.
\end{lem}

\begin{proof}
  Take any point $p \in M$ and a ray $\sigma : [\,0,+\infty\,) \to M$
  from $p$ asymptotic to $\gamma_+$.
  For any $s > 0$, the directional derivatives of $b_+$ to the two
  opposite directions at $\sigma(s)$ tangent to $\sigma$ are
  $-1$ and $1$ respectively.
  Since $b_- = -b_+$ and by Lemma \ref{lem:diff},
  a ray from $\sigma(s)$ asymptotic to $\gamma_-$ is unique
  and contains $\sigma([\,0,s\,])$.
  By the arbitrariness of $s > 0$,
  $\sigma$ extends to a straight line bi-asymptotic to $\gamma$.
  Namely, for a given point $p \in M$, we have a straight line
  $\sigma_p$ passing through $p$ and bi-asymptotic to $\gamma$.
  By Lemma \ref{lem:bfunc}, any ray from a point in $\sigma_p$
  asymptotic to $\gamma_\pm$ is a subray of $\sigma_p$.
  In particular, $\sigma_p$ is unique (upto parameters) for a given $p$.
  $M$ is covered by $\{\sigma_p\}_{p \in M}$ and
  this completes the proof.
\end{proof}

\begin{lem} \label{lem:bsubh}
  Assume that $M$ satisfies $\BG(0)$ at any point on a ray $\gamma$ in $M$.
  Then, the Busemann function $b_\gamma$ is $\mathcal{E}$-subharmonic.
\end{lem}

See Definition \ref{defn:subh} for the definition of
$\mathcal{E}$-subharmonicity.

\begin{proof}
  We take a sequence $t_i \to +\infty$, $i=1,2,\dots$.
  Since $r_{\gamma(t_i)}$, $b_\gamma$ are $1$-Lipschitz, they are
  $\Hm^n$-a.e.~differentiable.
  Let $x \in M^*$ be any point where $r_{\gamma(t_i)}$ and $b_\gamma$ are
  all differentiable.
  We have a unique minimal geodesic $\sigma_{x,i}$ from $x$ to
  $\gamma(t_i)$ and $\nabla r_{\gamma(t_i)}(x)$ is tangent to it.
  A ray $\sigma_x$ from $x$ asymptotic to $\gamma$ is unique
  and $-\nabla b_\gamma(x)$ is tangent to it.
  Since $\sigma_{x,i} \to \sigma_x$ as $i \to \infty$,
  we have $\nabla r_{\gamma(t_i)}(x) \to -\nabla b_\gamma(x)$.
  Therefore, the dominated convergence theorem
  and Corollary \ref{cor:LapCompform} show that for
  any $u \in C^\infty_0(M^*)$ with $u \ge 0$,
  \begin{align*}
    -\int_{M^*} \langle \nabla b_\gamma,\nabla u\rangle \; d\Hm^n
    &= \lim_i \int_{M^*} \langle \nabla r_{\gamma(t_i)},\nabla u\rangle
    \; d\Hm^n\\
    &\ge -(n-1) \lim_i \int_{M^*} \frac{u}{r_{\gamma(t_i)}} \; d\Hm^n = 0.
  \end{align*}
  This completes the proof.
\end{proof}

\begin{rem}
  In general, $b_\gamma$ is not of DC and $\blap b_\gamma$ does not exist
  as a Radon measure.
\end{rem}

\begin{proof}[Proof of Theorem \ref{thm:splitting}]
  By Lemma \ref{lem:bsubh},
  $b := b_+ + b_-$ is $\mathcal{E}$-subharmonic.
  It follows from the triangle inequality that $b \le 0$.
  We have $b \circ \gamma \equiv 0$ by the definition of $b$.
  The maximum principle (Lemma \ref{lem:maxprin}) proves $b \equiv 0$.
  Lemma \ref{lem:bsplitting} implies the theorem.
\end{proof}

\begin{proof}[Proof of Corollary \ref{cor:splitting}]
  We denote by $\Delta$ the usual Laplacian induced from
  the $C^\infty$ Riemannian metric on $M \setminus S_M$.
  It follows from $b_+ + b_- = 0$ that $b_+$ is $\mathcal{E}$-subharmonic and
  $\mathcal{E}$-superharmonic, so that
  $b_+$ is a weak solution of $\Delta u = 0$ on $M \setminus S_M$.
  By the regularity theorem of elliptic differential equation,
  $b_+$ is $C^\infty$ on $M \setminus S_M$ and satisfies
  $\Delta b_+ = 0$ pointwise on $M \setminus S_M$.
  By using Weitzenb\"ock formula and by $\Ric(\nabla b_+,\nabla b_+) \ge 0$,
  the Hessian of $b_+$ vanishes on $M \setminus S_M$, namely
  $b_+$ is a linear function along any geodesic in $M \setminus S_M$.
  Since any geodesic joining two points in $M \setminus S_M$ is
  contained in $M \setminus S_M$, the set of geodesic segments
  in $M \setminus S_M$ is dense in the set of all geodesic segments.
  Therefore, $b_+$ is linear along any geodesic in $M$.
  Since $M$ is covered by straight lines bi-asymptotic to $\gamma$,
  $b_+$ is averaged $D^2$ in the sense of \cite{Mk:splitting}.
  The corollary follows from
  Theorem A of \cite{Mk:splitting}.
\end{proof}

\bibliographystyle{amsplain}
\bibliography{all}

\nocite{BGP, Ot:secdiff, Pr:DC, OS:rstralex}
\nocite{Pt:harmAlex, Pt:subharmAlex}

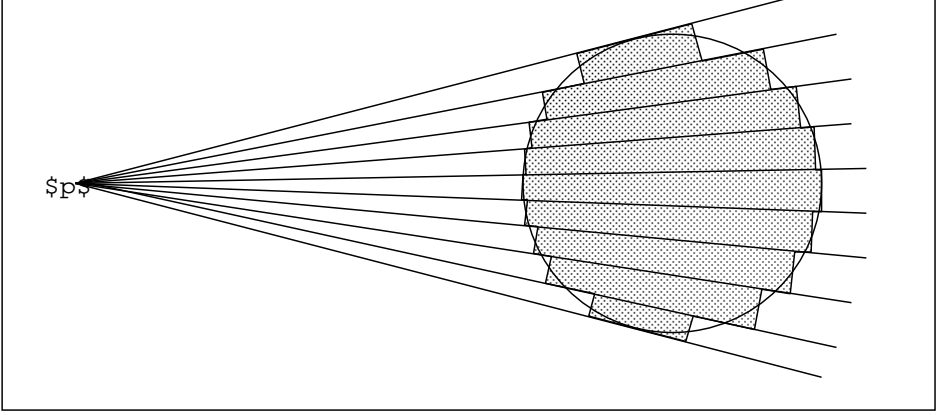
\end{document}